\documentclass{article}
\usepackage[utf8]{inputenc}
\usepackage[english]{babel}
\usepackage{amsmath,amssymb,bbm,latexsym,tikz-cd,amsthm}
\usepackage[utf8]{inputenc}
\usepackage[T2A]{fontenc}

\title{Persistence of sub-chain groups}
\author{Fang Sun, Shengwen Xie, Xuezhi Zhao}
\date{September 2021}
\newtheorem{theorem}{Theorem}
\newtheorem{prop}[theorem]{Proposition}
\newtheorem{remark}[theorem]{Remark}
\newtheorem{lemma}[theorem]{Lemma}
\newtheorem{definition}[theorem]{Definition}
\newtheorem{corollary}[theorem]{Corollary}
\newtheorem{exmp}[theorem]{Example}
\newenvironment{sproof}{%
  \proof}{\endproof}

\numberwithin{theorem}{section}

\begin{document}

\maketitle
\section{abstract}
In this work, we present a generalization of extended persistent homology  to filtrations of graded sub-groups by defining relative homology in this setting. Our work provides a more comprehensive and flexible approach to get an  algebraic invariant overcoming the limitations of the standard approach. The main contribution of our work is the development of a stability theorem for extended persistence modules using an extension of the definition of interleaving and the rectangle measure. This stability theorem is a crucial property for the application of mathematical tools in data analysis. We apply the stability theorem to extended persistence modules obtained from extended path homology of directed graphs and extended homology of hypergraphs, which are two important examples in topological data analysis. 
\section{Introduction}

The study of graded subgroups and their homology is motivated by at least two subjects in TDA: path homology of directed graphs and homology of hypergraphs.

Path homology is a novel algebraic invariant for directed graphs (digraphs) developed by Grigoryan, Lin, Muranov and Yau (c.f. \cite{grigor2020path}). It is capable of extracting higher dimensional features from a graph, while taking direction into account. See \cite{Chowdhury18} p.1154 Section 2.4 for an example of two directed graphs whose underlying undirected graphs are identical having distinct path homology groups. In addition, path homology satisfy certain functorial properties one would expect from a "natural" homology theory on digraphs. Namely, it is natural with respect to morphisms of digraphs and satisfy a K\"unneth Formula for product of digraphs. There is even a homotopy theory on digraphs that is compatible with path homology. This makes path homology an ideal feature for studying digraphs, assuming an efficient algorithm is available.

\section{Preliminaries}
In what follows, we fix a field $k$, and assume all vector spaces to be over $k$, and all maps between vector spaces to be $k$-linear.
\subsection{Filtration and Basis}
\subsubsection{Compatible Basis}
Suppose we are given a $k$ vector space $C$ and a filtration
\begin{equation*}
    D^1 \subseteq D^2 \subseteq \cdots \subseteq D^N \subseteq C
\end{equation*}
of subspaces (we do not assume this filtration to be exhaustive, i.e., $D^N = C$). A \textit{compatible basis} of the filtration $\{D^i\}_{1 \leq i \leq N}$ is an ordered basis $\{e^1,e^2,\cdots,e^m\}$ of $D^N$ such that there exists integers $ 1 \leq i_1 \leq i_2 \cdots \leq i_N = m$ so that for each $1 \leq j \leq N$, $D^j$ is spanned by the first $i_j$ members of $e^1,e^2,\cdots,e^m$, i.e., $D^j = span\{e^1,e^2,\cdots, e^{i_j} \}$. 

An \textit{upper unitriangular} matrix is an upper triangular matrix whose entries on the diagonal are all 1's.  A basis change between ordered basis of finite dimensional vector spaces is called a \textit{left-to-right} addition if its transition matrix is upper unitriangular. The following is an easy observation:
\begin{prop}
Let $\{e^1,e^2,\cdots,e^m \}$ be a compatible basis of the filtration $\{D^i\}_{1 \leq i \leq N}$ on $C$. If $\{\tilde{e}^1,\tilde{e}^2,\cdots,\tilde{e}^m \}$ is obtained from  $\{e^1,e^2,\cdots,e^m \}$ by a
left-to-right addition, then $\{\tilde{e}^1,\tilde{e}^2,\cdots,\tilde{e}^m \}$ is also compatible.
\end{prop}

\subsubsection{Tail Position and Height}
Choose and fix a compatible basis $\{e^1,e^2,\cdots,e^m \}$ of $\{D^i\}_{1 \leq i \leq N}$. For any $x \in D^N$, define its \textit{tail position} by
\begin{equation*}
    tp(x)=
    \begin{cases}
    \max \{ 1 \leq i \leq m | x_i \neq 0  \} & \mbox{if } x= \sum_{i=1}^{m} x_i e^i \neq 0 \\
    0 & \mbox{if } x=0
    \end{cases}
\end{equation*}
If $x \in C - D^N$, $tp(x)$ is not defined. Suppose $\{\tilde{e}^1,\tilde{e}^2,\cdots,\tilde{e}^m \}$ is obtained from $\{e^1,e^2,\cdots,e^m \}$ by a left-to-right addition, then $tp(\tilde{e}^i)=i,1 \leq i \leq n $.

Given a filtration $\{D^i\}_{1 \leq i \leq N}$ on $C$, for any $x \in C$, define the its \textit{height} $ht(x)$ by
\begin{equation*}
    ht(x)=
    \begin{cases}
    \min\{ 1 \leq i \leq N | x \in D^i \} & \mbox{if } x \in D^N \\
    \infty & \mbox{otherwise}
    \end{cases}
\end{equation*}

For $x \in D^N$, its height can be read from the tail position with respect to a compatible basis, i.e., $ht(x) = ht(e^{tp(x)})$

In general, there is no guarantee that linear combinations preserve height. For left-to-right addition on compatible basis, however, we have the following proposition:
\begin{prop}
If $\{e^1,e^2,\cdots,e^m\}$ is a compatible basis of $\{D^i\}_{1 \leq i \leq N}$ and $\{\tilde{e}^1,\tilde{e}^2,\cdots,\tilde{e}^m \}$ is obtained from $\{e^1,e^2,\cdots,e^m \}$ by a
left-to-right addition, then we have $ht(\tilde{e}^i)=ht(e^i),1 \leq i \leq m$.
\end{prop}
\begin{proof}
This follows from $tp(\tilde{e}^i)=i=tp(e^i)$.
\end{proof}

\subsection{Graded Subgroups of a Chain Complex}\label{Graded Subgroups of a Chain Complex}
\subsubsection{Definition}
Let 
\begin{equation*}
    \cdots \xrightarrow{\partial_{n+1}} C_n \xrightarrow{\partial_{n}} C_{n-1} \xrightarrow{\partial_{n-1}} \cdots
\end{equation*}
be a chain complex of $k$-vector spaces. A \textit{graded subgroup} of $\{C_\ast, \partial_\ast \}$ is just a family of $\{ D_\ast \}$ of $k$-vector spaces such that $D_p \subset C_p$ for all $p$. If $\partial D_p \subset D_{p-1}$ then $\{ D_\ast \}$ would be a subcomplex, but we do not assume this in general.

To define a proper notion of homology for graded subgroups, ones has to construct a chain complex from $\{ D_\ast \}$. It seems reasonable to either enlarge $\{ D_\ast \}$ or excise it. As it turns out, both approaches produce the same homology groups. We now give a detailed treatment of these notions.

For each $p$, define $S_p =S_p(D_\ast ;C_\ast) = D_p + \partial_{p+1} D_{p+1}$ and $I_p = I_p(D_\ast ;C_\ast) = D_p \cap \partial^{-1}_p D_{p-1}$. Then $S_\ast$ (resp. $I_\ast$) is a subcomplex of $C_\ast$, and is called the \textit{supremum} (resp. \textit{infimum}) \textit{chain complex} of $\{ D_\ast \}$ in $\{ C_\ast, \partial _\ast \}$. 
\begin{remark} \label{RemarkGradedSubgroups}
 It can be shown (\cite{HypergraphHomology} Proposition 2.1, 2.2) that the $S_\ast$ is the minimal subcomplex containing $\{ D_\ast \}$, while $I_\ast$ is the maximal subcomplex contained in $\{ D_\ast \}$. In particular, replacing $C_\ast$ by a larger chain complex (containing it) does not affect $S_\ast$ or $I_\ast$.
\end{remark}

Define the homology of $S_\ast$ (resp. $I_\ast$) by $H_\ast^{\sup}(D_\ast; C_\ast)$ (resp.$H_\ast^{\inf}(D_\ast; C_\ast)$).
We have the following result (\cite{HypergraphHomology} Propposition 2.4):
\begin{prop}
The inclusion of $I_\ast \hookrightarrow S_\ast$  is a chain map that induce isomorphisms $H_\ast^{\inf}(D_\ast; C_\ast) \xrightarrow[]{} H_\ast^{\sup}(D_\ast; C_\ast)$. 
\end{prop}

The above isomorphism is natural in the following sense:
Suppose we have two graded subgroups $\{ D_\ast \}$ and $\{ D'_\ast \}$ of $\{C_\ast, \partial_\ast \}$, and $D_p \subset D'_p$ for all $p$. Let $S_\ast, I_\ast$ (resp. $S'_\ast, I'_\ast$) be the supremum and infimum subcomplex of $D_\ast$ (resp. $D'_\ast$), then $S_\ast$, $I_\ast$ embeds in $S'_\ast$, $I'_\ast$ respectively, such that the following diagram

\begin{tikzcd}
 H^{\inf}_\ast (D_\ast; C_\ast) \arrow[r] \arrow[d] & H^{\sup}_\ast (D_\ast; C_\ast) \arrow[d]\\
 H^{\inf}_\ast (D'_\ast; C_\ast) \arrow[r]         & H^{\sup}_\ast (D'_\ast; C_\ast)
\end{tikzcd} \\
commutes, where all homomorphisms are induced by inclusions.

There are at least two settings where graded subgroups arise naturally. We present both of them in the following.

\subsubsection{Path Homology}
The first case concerns the notion of path homology of directed graphs (digraphs). Formally, a (finite) digraph is a pair $G=(X,E)$ where $X$ is a finite set and $E \subset X \times X$. We deal exclusively with digraphs without self-loops. 

Given a non-negative integer $p$, an \textit{elementary p-path} on the set $X$ is a sequence $x_0\cdots x_p$ of $p+1$ elements of $X$. Denote by $\Lambda_p = \Lambda_p(X)$ the $k$ vector space of all formal linear combinations of elementary $p$-paths on $X$. An elementary path $x_0\cdots x_p$ as an element of $\Lambda_p$ is denoted $e_{x_0\cdots x_p}$. Define the boundary operator $\overline{\partial} : \Lambda_p \to \Lambda_{p-1}$ by
\begin{equation*}
    \overline{\partial} (e_{x_0\cdots x_p}) = \sum_{i=0}^{p} (-1)^i e_{x_0\cdots\hat{x_i} \cdots x_p}
\end{equation*}
where $\hat{x_i}$ denotes omission as usual. It can be shown (\cite{grigor2020path} Lemma 2.1) that ${\overline{\partial}}^2 =0$, thus $\{\Lambda_\ast\}$ is a chain complex.

An elementary path $x_0\cdots x_p$ is called \textit{regular} if $x_{k-1} \neq x_k, 1 \leq k \leq p$. Define $\mathcal{R}_p = \mathcal{R}_p(X)$ (resp. $\mathcal{I}_p=\mathcal{I}_p(X)$) as the subspace of $\Lambda_p$ consisting of the formal linear combinations of regular (resp. irregular) elementary $p$-paths. It is not hard to check that $\mathcal{I}_\ast$ is a subcomplex of $\Lambda_\ast$, so the quotient chain complex $(\Lambda / \mathcal{I})_\ast$ is well-defined. Since $(\Lambda / \mathcal{I})_p$ and $\mathcal{R}_p$ are canonically isomorphic, we obtain a chain complex $\{\mathcal{R}_\ast, \partial_\ast\}$ via this identification.

The digraph structure enters the scene in the following way: an elementary regular path $x_0\cdots x_p$ is called \textit{allowed} if each $(x_{k-1},x_k)$ belongs to the edge set $E$. In other words, it is a head-to-tail concatenation of arrows. The set of formal linear combinations of allowed elementary $p$-paths of $G$ is denoted $\mathcal{A}_p = \mathcal{A}_p(G)$. Noting that $\mathcal{A}_\ast$ is a graded subgroup of $\{\mathcal{R}_\ast, \partial_\ast\}$, we define the $p$-dimensional path homology of $G$ as $H_p(G) := H_p^{\sup}(\mathcal{A}_\ast; \mathcal{R}_\ast) \cong H_\ast^{\inf}(\mathcal{A}_\ast; \mathcal{R}_\ast)$.

Given $G=(X,E)$ and $G^t=(X,E')$ with $E \subset E'$, we have an embedding $\mathcal{A}_p(G) \subset \mathcal{A}_p(G^t)$ for each $p$, and thus an induced homomorphism $H_\ast(G) \to H_\ast(G^t)$.

\subsubsection{Hypergraph Homology}
The second case concerns homology of hypergraphs. Let $V$ be a finite set. The power set $P(V)$ of $V$ is the collection of all non-empty subsets of $V$. A \textit{hypergraph} $\mathcal{H}$ on $V$ is just a subset of $P(V)$. Elements of $\mathcal{H}$ are called \textit{hyperedges}. If $\tau \subset \sigma \in \mathcal{H}$ implies $\tau \in \mathcal{H}$, then $\mathcal{H}$ would be an abstract simplicial complex, but we do not assume this in general.

Any hypergraph $\mathcal{H}$ can be embedded in a simplicial complex $K$. The most economical choice being $K_{\mathcal{H}} = \{\tau | \tau \subset \sigma \mbox{ for some } \sigma \in \mathcal{H}\}$. Actually, $K_{\mathcal{H}}$ is minimal among all such $K$. 

For a simplicial complex $L$, let $\Delta_\ast(L)$ denote the (oriented) simplicial chain complex of $L$. Given a hypergraph $\mathcal{H}$, define $\Delta_\ast(\mathcal{H})$ as the graded subgroup of $\Delta_\ast(K_{\mathcal{H}})$ spanned by the hyperedges of $\mathcal{H}$, and define the \textit{embedded homology} of $\mathcal{H}$ as $H_\ast(\mathcal{H})= H^{\sup}_\ast(\Delta_\ast(\mathcal{H}),\Delta_\ast(K_{\mathcal{H}})) \cong H^{\inf}_\ast(\Delta_\ast(\mathcal{H}),\Delta_\ast(K_{\mathcal{H}}))$. When $\mathcal{H}$ is embedded in a simplicial complex $K$, there are canonical embeddings of $\Delta_\ast(K_{\mathcal{H}})$ and $\Delta_\ast(\mathcal{H})$ into $\Delta_\ast(K)$. By Remark \ref{RemarkGradedSubgroups}, we could replace $K_{\mathcal{H}}$ by any such $K$ in the definition of $H_\ast(\mathcal{H})$. This enables us to 
formulate naturality for hypergraph homology. Given hypergraphs $\mathcal{H}' \subset \mathcal{H}$, choose a embedding of $\mathcal{H}$ in a simplicial complex $K$. There is an induced homomorphism
\begin{equation*}
    H_\ast(\mathcal{H}') = H_\ast(\Delta_\ast(\mathcal{H}'),\Delta_\ast(K)) \rightarrow H_\ast(\Delta_\ast(\mathcal{H}),\Delta_\ast(K)) = H_\ast(\mathcal{H})
\end{equation*}
Note that this homomorphism is independent of the choice on $K$.

\subsection{Persistent Homology}

Here is a brief review of the standard persistent homology from an algebraic perspective.

\subsubsection{Persistence Module and Interval Decomposition}
For our purpose, a \textit{persistence module} $V= (\{V^i\}, \{ \phi^i \})$ is a sequence of $k$-vector spaces and $k$-linear maps
\begin{equation*}
    \cdots \xrightarrow{\phi^{i-1}} V^i \xrightarrow{\phi^{i}} V^{i+1} \rightarrow \cdots
\end{equation*}
A \textit{morphism} of persistence module from $(\{V^i\}, \{ \phi^i \})$ to $(\{W^i\}, \{ \psi^i \})$ is a sequence of linear maps $\eta^i : V^i \rightarrow W^i$ such that $\psi^i \circ \eta^i = \eta^{i+1} \circ \phi^i$. If each $\eta^i$ is an isomorphism, we call it an \textit{isomorphism of persistence modules}.
Given a family $\{ V_\lambda \}$ of persistence modules, define their \textit{direct sum} $\oplus_\lambda V_\lambda$ by $(\oplus_\lambda V_\lambda^i, \oplus_\lambda \phi_\lambda^i )$. \\

We are mainly interested in decomposing a persistence module into simple building blocks up to isomorphism. The simplest non-trivial modules are the interval modules.

For an interval $I$, the \textit{interval module} $k_I$ is defined by
\begin{equation*}
    k_I^m=
    \begin{cases}
    k & m \in I \\
    0 & otherwise
    \end{cases}
\end{equation*}
with identity map joining nonzero vector spaces (the other maps are automatically zero).
Most persistence modules we encounter in practice are decomposable into interval modules:
\begin{prop}{(\cite{botnan2020decomposition} Theorem 1.2)}
If a persistence module $V$ is of pointwise finite-dimensional, i.e., $\dim V^i < \infty$ for all $i$, then $V$ is isomorphic to a direct sum of interval modules. Moreover, such decomposition is unique.
\end{prop}

The collection of intervals in the decomposition of $V$ is called the \textit{persistent barcode} of $V$.
\subsubsection{Persistent Homology}

In topological data analysis, persistence modules usually arise as homology groups of filtrations of subcomplexes. Given a chain complex $\{C_\ast, \partial_\ast \}$ and a filtration
\begin{equation}\label{eq:1}
    C^1_\ast \hookrightarrow \cdots \hookrightarrow C^i_\ast \hookrightarrow C^{i+1}_\ast \hookrightarrow \cdots \hookrightarrow C^{N}_\ast
\end{equation}
of subcomplexes. The inclusions induce a persistence module on homology
\begin{equation*}
    \cdots \rightarrow H_p(C^i_\ast) \rightarrow H_p(C_\ast^{i+1}) \rightarrow \cdots
\end{equation*}
for each dimension $p$. These are called the \textit{persistent homology} (PH) of (\ref{eq:1}).

In practice, the filtration (\ref{eq:1}) usually comes from a filtration
\begin{equation*}
    \cdots \hookrightarrow K^i \hookrightarrow K^{i+1} \hookrightarrow \cdots
\end{equation*}
of subcomplexes of an ambient simplicial complex $K$, as the (oriented) simplicial chain complexes. The most common scenario where this arises is where we are given a real valued function $f : K^{(0)} \rightarrow \mathbb{R}$ defined on the vertex set of a simplicial complex $K$. Let $ -\infty = a_0 < a_1 <a_2 < \cdots < a_n$ be the values of $f$. The sublevel subcomplexes $K_i = \{ \sigma \in K | f(v) \leq a_i \mbox{ for all } v \in \sigma \}$ constitutes a filtration. The persistent homology of their simplicial chain complexes is called the persistent homology of $K$ with respect to $f$.

\subsubsection{The Standard PH Algorithm}\label{StandardPHAlgorithm}
A crucial step of topological data analysis is obtaining the interval decomposition of the PH of (\ref{eq:1}). We shall give a sketch description of the standard algorithm for this task, for later we will detail a generalized version of it.

We need a few algebraic preparations. For each nonzero column $j$ of a matrix $A$ (with entries in $k$), define $low_j A$ as the order of the row with the lowest nonzero element of $\mbox{col}_j$. We say a matrix $A$ is \textit{reduced} if $low_i \neq low_j$ for $i \neq j, \mbox{col}_i \neq 0, \mbox{col}_j \neq 0$, and call the position of those $low_j$ the \textit{pivot} of $A$. Any matrix can be turned into a reduced one by left-to-right additions on the columns, and the location of pivots are independent of the particular choice of left-to-right additions (\cite{CT} p.154).

We may assume the filtration (\ref{eq:1}) is exhaustive, for otherwise we could replace $C_\ast$ by $C_\ast^N$. The input of the algorithm is a compatible basis $\{e_p^1,e_p^2,\cdots,e_p^{m_p}\}$ of the filtration
\begin{equation*}
    \cdots \hookrightarrow C^i_p \hookrightarrow C^{i+1}_p \hookrightarrow \cdots
\end{equation*}
for each $p$. In the simplicial setting, the basis elements corresponds to simplexes. We find the matrix of the boundary operator $\partial_{p+1}$ with respect to $\{e_p^1,e_p^2,\cdots,e_p^{m_p}\}$ and $\{e_{p+1}^1,e_{p+1}^2,\cdots,e_{p+1}^{m_{p+1}}\}$. Perform a reduction of $\partial_{p+1}$ by left-to-right additions and collect the pivot positions from the reduced matrix. Each pivot position $low_j=i$ produce a \textit{pairing} between $e_p^i$ and $e^j_{p+1}$. These pairings are all one needs to write down the interval decomposition of the persistent homology. To be precise, each unpaired $e_p^i$ contribute an interval $[ht(e_p^i), +\infty)$ to $H_p$, while each pair $(e_p^i, e^j_{p+1})$ contributes an interval $[ht(e_p^i), ht(e^j_{p+1}))$ to $H_p$.

\section{PH for Filtration of Graded Subgroups}
\subsection{Background}
Suppose now we have a filtration
\begin{equation*}
    \cdots \hookrightarrow D^i_\ast \hookrightarrow D^{i+1}_\ast \hookrightarrow \cdots
\end{equation*}
of graded subgroups of a chain complex $\{C_\ast, \partial_\ast \}$. Define the persistent ($p$-dimensional) homology of this filtration to be the persistent module:
\begin{equation*}
    \cdots \rightarrow H^{sup}_p(D^i_\ast) \rightarrow H^{sup}_p(D_\ast^{i+1}) \rightarrow \cdots
\end{equation*}
By the discussion in Section \ref{Graded Subgroups of a Chain Complex}, this is equivalent (up to isomorphism) to
\begin{equation*}
    \cdots \rightarrow H^{inf}_p(D^i_\ast) \rightarrow H^{inf}_p(D_\ast^{i+1}) \rightarrow \cdots
\end{equation*}

Our goal is to give an algorithm for computing the interval decomposition of this persistence module. The input should be a compatible basis $\{e_p^1,e_p^2,\cdots,e_p^{m_p}\} $ of the filtration
\begin{equation*}
    \cdots \hookrightarrow D^i_p \hookrightarrow D^{i+1}_p \hookrightarrow \cdots \hookrightarrow D^N_p 
\end{equation*}
for each $p$, a basis of $C_p$ extending $\{e_p^1,e_p^2,\cdots,e_p^{m_p}\} $, together with data recording the the behavior of $\partial$ with respect to these basis. Note that the extension is necessary, for $\partial D^N_{p+1}$ may not be contained in $D^N_p$. As is expected, not all of the extended basis are relevant: we merely need those that appear in the boundary of $e_p^i$'s, which in practice might be a tiny fraction.

It might appears economical to deal with infimum chain complexes, since they are smaller. An algorithm following this line of thought is detailed in \cite{Chowdhury18}. This approach requires as a initial step the computation of a compatible basis for the filtration
\begin{equation*}
    \cdots \hookrightarrow I^i_\ast \hookrightarrow I^{i+1}_\ast \hookrightarrow \cdots
\end{equation*}
Once such compatible basis is obtained, one could apply the standard PH algorithm. The initial step requires not only column operations, but also row operations to keep track of the updated basis. Another drawback is the assumption that the filtration $\{ D^i_\ast \}$ be exhaustive, i.e., $\cup_i D^i_p=C_p$ for all $p$. If this is not satisfied then one has to add $C_p$ as $D^{N+1}_p$. Thus in effect one has to deal with the entirety of $C_p$. 

The method we are about to present uses supremum complexes instead. We use column operations only, and these operations need not be recorded.

\begin{remark}
 The method of \cite{Chowdhury18} can be improved in the following way: one could add $S^N_\ast$ as $D^{N+1}_\ast$, and replace $C_\ast$ by $S^N_\ast$. This would significantly reduce the size of basis to be processed. However, it requires another step of computing $S^N_\ast$, and is still far more complicated than the method using supremum complexes.
\end{remark}
 
\subsection{Computing by supremum complex}\label{ComputePHbySup}
We now present a method based on supremum complexes. Instead of computing a compatible basis of 
\begin{equation*}
    \cdots \hookrightarrow S^i_\ast \hookrightarrow S^{i+1}_\ast \hookrightarrow \cdots
\end{equation*}
and invoking the standard PH algorithm, we will mimic the construction in said algorithm.

For each $p$, let $\{e_p^1,e_p^2,\cdots,e_p^{m_p}\}$ be a compatible basis of the filtration $\{D_p^i\}_{1\leq i\leq N}$. In the rest of this section, all tail positions are with respect to $\{e_p^1,e_p^2,\cdots,e_p^{m_p}\}$ and all heights are with respect to $\{D_p^i\}_{1\leq i\leq N}$. 

Let $\{e_p^1,e_p^2,\cdots,e_p^{m_p}\} \cup B_p$ be a (extended) basis of $C_p$, and let
$\{ \epsilon_p^1, \epsilon_p^2, \cdots , \varepsilon^{n_p}_p \}$ be the members of $B_p$ that appear in the expression of $\partial e^j_{p+1} $'s with respect to $\{e_p^1,e_p^2,\cdots,e_p^{m_p}\} \cup B_p$. Denote by $A_{p+1}$ the matrix of $\partial_{p+1|D_{p+1}^N} $ with respect to $\{ e^1_{p+1}, \cdots e^{m_{p+1}}_{p+1} \}$ and $\{e_p^1,e_p^2,\cdots,e_p^{m_p}, \epsilon_p^1, \epsilon_p^2, \cdots , \epsilon^{n_p}_p\}$, ordered as such. Since $S_{p+1}^N = D_{p+1}^N + \partial D_{p+2}^N$ and $\partial \circ \partial = 0$, the boundary map $S_{p+1}^N \rightarrow S_p^N$ of the supremum complex is encoded in $A_{p+1}$. In practice (e.g. hypergraphs and path complexes), both $\{ \epsilon_p^1, \epsilon_p^2, \cdots , \epsilon^{n_p}_p \}$ and $A_{p+1}$ should be readily read off any appropriate representation of the data.

Next we reduce $A_{p+1}$ by performing left-to-right additions on its columns. Denote the reduced matrix as $\tilde{A}_{p+1}$, and the resulted new basis of $D^N_{p+1}$ as $\{ \tilde{e}^1_{p+1}, \cdots \tilde{e}^{m_{p+1}}_{p+1} \}$, so $\mbox{col}_j\tilde{A}_{p+1}$ encodes $\partial \tilde{e}^j_{p+1}$. If we merely need the persistent barcodes, then the expression of $\{ \tilde{e}^1_{p+1}, \cdots \tilde{e}^{m_{p+1}}_{p+1} \}$ with respect to $\{ {e}^1_{p+1}, \cdots {e}^{m_{p+1}}_{p+1} \}$ need not be stored. In other word, we do not have to record the column additions performed.

The pivot positions in $\tilde{A}_{p+1}$ defines a pairing: we say $(e^i_p,e^j_{p+1})$ is a \textit{pair} if the pivot of $\mbox{col}_j\tilde{A}_{p+1}$ is at the $i$-th row. Equivalently, this is saying that $tp(\partial \tilde{e}^j_{p+1}) = i$. Since pivot positions are independent of the particular reduction strategy, so are the pairings.

Similar to pairing in standard PH, an element of $\cup_p \{e^1_p,\cdots,e^{m_p}_p \}$ appears in at most one pair, this will follow from the following proposition:
\begin{prop} \label{Clearing}
If $(e^i_p,e^j_{p+1})$ is a pair, then $\mbox{col}_i\tilde{A}_{p}=0$.
\end{prop}
 \begin{proof}
 The pairing condition means that $\partial \tilde{e}^j_{p+1} = a_1 e_p^1 + a_2 e_p^2 + \cdots +a_i e_p^i$ where $a_i \neq 0$. Since $\partial (\partial \tilde{e}^j_{p+1}) = a_1 \partial e_p^1 + a_2 \partial e_p^2 + \cdots +a_i \partial e_p^i = 0$, there is a reduction $\tilde{A}_{p}$ of $A_p$ such that $\mbox{col}_i\tilde{A}_{p}=0$. The invariance of pivot positions means this is true for all possible reductions.
 \end{proof}
 
 \begin{remark}
  The pairing defined above obviously generalizes pairing in standard PH. A few notable distinctions are:
  
 (1) In standard PH, when $e^i_p,e^j_{p+1}$ are paired, we always have $ht(e^i_p) \leq ht(e^j_{p+1})$. This is no longer true for generalized pairing, for $D^i_\ast$'s are not closed under the boundary map.
 
 (2) In standard PH, if $\mbox{col}_j\tilde{A}_{p+1} \neq 0$, then $e^j_{p+1}$ is always paired (with some $e^i_p$). This is not true for the generalized case, since the pivot of $\mbox{col}_j\tilde{A}_{p+1}$ may fall into the rows corresponding to the $\epsilon_p$'s.
 \end{remark}
 
 \begin{remark}
  Proposition \ref{Clearing} implies that the clearing method introduced in \cite{chen2011persistent} is applicable in our setting. That is, we may reduce the matrices $A_p$ in decreasing order of dimension (assuming the chain complex is finite dimensional), and omit the reduction of $col_i A_p$ if $e_p^i$ is already paired.  
 \end{remark}
 \
 The desired PH can be read from the pairing and height of basis elements. To be precise, we have:
 \begin{theorem}
 The pairing determines the interval decomposition of the persistence module
 \begin{equation}\label{eq:2}
    \cdots \rightarrow H^{sup}_p(D^i_\ast) \rightarrow H^{sup}_p(D_\ast^{i+1}) \rightarrow \cdots
\end{equation}
in the following way:

(i) each unpaired $e_p^i$ such that $\mbox{col}_i\tilde{A}_{p} = 0$ contributes an interval $[ht(e_p^i), + \infty)$.

(ii) each pair $(e^i_p,e^j_{p+1})$ such that $ht(e^i_p) < ht(e^j_{p+1})$ contributes an interval $[ht(e_p^i), ht(e_{p+1}^j))$ (see Figure).
 \end{theorem}

 \begin{proof}
 We will construct a morphism from the direct sum of all the interval modules mentioned in (i) and (ii) to (\ref{eq:2}). By the universal property of direct sums, it suffice to define it on each summand. A morphism from an interval module the form $k_{[a,b)}$ or $k_{[a, + \infty)}$ is determined by the image of identity element of the vector space $k$ at $a$. We call this element the \textit{initial 1} of that interval module. 
 
 For each $e_p^i$ in case (i), map the initial $1$ of the corresponding interval module to the homology class of $H^{sup}_p(D^{ht(e_p^i)}_\ast)$ represented by the cycle $\tilde{e}_p^i$. For each pair $(e^i_p,e^j_{p+1})$ in case (ii), map the initial $1$ to the homology class of $H^{sup}_p(D^{ht(e_p^i)}_\ast)$ represented by $\partial \tilde{e}^j_{p+1}$. In case (ii), the homology class represented by $\partial \tilde{e}^j_{p+1}$ dies before entering the $ht(e_{p+1}^j)$-th stage. Thus we have a well-defined morphism $\Phi$ from the direct sum of the interval modules to (\ref{eq:2}).
 
 It remains to verify that $\Phi$ is an isomorphism at each level $i, 1 \leq i \leq N$. Let $\tilde{e}^{j_1}_p, \tilde{e}^{j_2}_p, \cdots , \tilde{e}^{j_r}_p, \partial \tilde{e}^{k_1}_{p+1}, \partial \tilde{e}^{k_2}_{p+1}, \cdots , \partial \tilde{e}^{k_s}_{p+1}$ be the images of initial $1$'s of intervals spanning $i$. It suffices to show that:\\
 
 \begin{flushleft}
 \textbf{Claim}: The elements $ [\tilde{e}^{j_1}_p], [\tilde{e}^{j_2}_p], \cdots , [\tilde{e}^{j_r}_p], [\partial \tilde{e}^{k_1}_{p+1}], [\partial \tilde{e}^{k_2}_{p+1}], \cdots , [\partial \tilde{e}^{k_s}_{p+1}]$ form a basis of $H^{sup}_p(D^i_\ast)$, where [$\alpha$] denotes the homology class represented by $\alpha$.\\\
 
 \textbf{Proof of the Claim}:
 \end{flushleft}
 We will assume working with $\mathbb{Z}_2$ coefficients for simplicity, the necessary modifications for general field coefficients should be obvious.
 
 Let $S^i_\ast$ be the supremum complex of $D^i_\ast$, so that $H^{sup}_p(D^i_\ast) = H_p(S^i_\ast)$. Since $S_p^i=D_p^i + \partial D_{p+1}^i$ and $\partial^2 =0$, we have $H_p(S^i_\ast) = (\mbox{Ker}(D_p^i \xrightarrow{\partial} C_{p-1}) + \partial D_{p+1}^i)/ \partial D_{p+1}^i$.
 Denote the basis elements paired with ${e}^{k_1}_{p+1}, \cdots , {e}^{k_s}_{p+1}$ by ${e}^{l_1}_{p}, \cdots , {e}^{l_s}_{p}$. By assumption, the height of ${e}^{j_1}_p, \cdots , {e}^{j_r}_p, {e}^{l_1}_{p}, \cdots , {e}^{l_s}_{p}$ are at most $i$, while the height of ${e}^{k_1}_{p+1}, \cdots , {e}^{k_s}_{p+1}$ are strictly greater than $i$. \\
 
 Start with linear independence. Suppose there exist a linear combination
 \begin{equation*}
     x_1[\tilde{e}^{j_1}_p] + \cdots + x_r[\tilde{e}^{j_r}_p] + y_1[\partial \tilde{e}^{k_1}_{p+1}] + \cdots + y_s[\partial \tilde{e}^{k_s}_{p+1}] = 0
 \end{equation*}
 such that the coefficients are not all zero. Then the chain
 \begin{equation*}
     \alpha = x_1\tilde{e}^{j_1}_p + \cdots + x_r\tilde{e}^{j_r}_p + y_1\partial \tilde{e}^{k_1}_{p+1} + \cdots + y_s \partial \tilde{e}^{k_s}_{p+1}
 \end{equation*}
 belongs to $\partial D_{p+1}^i$. Since $tp(\tilde{e}^{j_q}_p)=j_q, 1 \leq q \leq r$, $tp( \partial \tilde{e}^{k_t}_{p+1})=l_t, 1 \leq t \leq s$, and $\{ j_q | 1 \leq q \leq r\} \cap \{ l_t | 1 \leq t \leq s\} = \varnothing$, we see that $tp(\alpha))$ is well defined and has to be among $j_1, \cdots ,j_r, l_1, \cdots, l_s$. On the other hand, since $\alpha \in \partial D_{p+1}^i$ and $D_{p+1}^i = \mbox{span} \{{e}_{p+1}^j | ht(e_{p+1}^j) \leq i \} = \mbox{span} \{\tilde{e}_{p+1}^j | ht(e_{p+1}^j) \leq i \}$, we have $\alpha \in \mbox{span} \{ \partial \tilde{e}_{p+1}^j | ht(e_{p+1}^j) \leq i \}$. In particular, $e_p^{tp(\alpha)}$ is paired with some $e_{p+1}^j$ with $ht(e_{p+1}^j) \leq i$. But ${e}^{j_1}_p, \cdots , {e}^{j_r}_p$ are not paired at all, while ${e}^{l_1}_p, \cdots , {e}^{l_s}_p$ are paired with ${e}^{k_1}_{p+1}, \cdots , {e}^{k_s}_{p+1}$ respectively, leading to a contradiction. This shows linear independence of $ [\tilde{e}^{j_1}_p], \cdots , [\tilde{e}^{j_r}_p], [\partial \tilde{e}^{k_1}_{p+1}], \cdots , [\partial \tilde{e}^{k_s}_{p+1}]$.
 
 It remains to show that $ [\tilde{e}^{j_1}_p], \cdots , [\tilde{e}^{j_r}_p], [\partial \tilde{e}^{k_1}_{p+1}], \cdots , [\partial \tilde{e}^{k_s}_{p+1}]$ span $H_p(S^i_\ast)$. Set $K=\mbox{Ker}(D_p^i \xrightarrow{\partial} C_{p-1})$ and $E=\{ \tilde{e}^{j_1}_p, \cdots , \tilde{e}^{j_r}_p, \partial \tilde{e}^{k_1}_{p+1}, \cdots , \partial \tilde{e}^{k_s}_{p+1} \}$. Note that $E \subseteq K$. It suffice to prove that any $x \in K$ can be linearly represented by elements of $E$ and elements of $\partial D_{p+1}^i$. Suppose we can find for each $x \neq 0 \in K$ a decomposition
 \begin{equation}\label{eq:3}
     x=x' +e +d
 \end{equation}
 where $x' \in K, tp(x') < tp (x), e \in \mbox{span}E$ and $d \in \partial D_{p+1}^i$, then an inductive argument will do. Since $x \in K$, $\mbox{col}_{tp(x)} {A}_p $ can be linearly represented by the columns to its left, and thus $\mbox{col}_{tp(x)} \tilde{A}_p = 0$ by the invariance of pivot positions. Also, $ht(\tilde{e}_p^{tp(x)}) = ht({e}_p^{tp(x)}) = ht(x) \leq i$. There are 3 cases:\\
 
 \textbf{Case (1)}: The element $e_p^{tp(x)}$ is not paired. Then $ht({e}_p^{tp(x)}) \leq i$ implies that $e_p^{tp(x)}$ is among $ {e}^{j_1}_p, \cdots , {e}^{j_r}_p$. A decomposition (\ref{eq:3}) is obtained by setting $e = \tilde{e}_p^{tp(x)}$, $x'=x-e$ and $d = 0$.
 
 \textbf{Case (2)}: The element $e_p^{tp(x)}$ is paired with $e_{p+1}^k$ and $ht(e_{p+1}^k) > i$. In this case, $ht({e}_p^{tp(x)}) \leq i$ and $ht(e_{p+1}^k) > i$ implies that $e_{p+1}^k$ is among ${e}^{k_1}_{p+1}, \cdots , {e}^{k_s}_{p+1}$. Thus $e = \partial \tilde{e}_{p+1}^{k}$, $x'=x-e$ and $d = 0$ is the desired decomposition.
 
 \textbf{Case (3)}: The element $e_p^{tp(x)}$ is paired with $e_{p+1}^k$ and $ht(e_{p+1}^k) \leq i$. Let $d= \partial \tilde{e}_{p+1}^k \in \partial D_{p+1}^i$, $e=0$ and $x'=x-d$. Now
 \begin{equation*}
 \begin{split}
     ht(d) & = ht (e_p^{tp(d)}) \\
      & = ht (e_p^{tp(x)}), \: \mbox{for}  \: tp(d) = tp(\partial \tilde{e}_{p+1}^k)= tp(x) 
 \end{split}
 \end{equation*}
Hence $ht(d) \leq i$ and $d \in D_p^i$. This means $x' \in D_p^i$. We also have $\partial x' = \partial x - \partial d = 0-0 =0$, whence $x' \in K$. Finally, $tp(d) = tp(x)$ implies $tp(x') < tp(x)$. This finishes the proof of the Claim.
 \end{proof}

\section{Extended Persistent Homology}

\subsection{Standard Extended PH}
An extended version of persistent homology for simplicial complexes is proposed in \cite{CohenSteiner2009ExtendingPH}. As we will see later, it extracts strictly more information than the unextended version without a drastic increase of computational complexity.

The input is usually a real valued function $f : K^{(0)} \rightarrow \mathbb{R}$. Let $ -\infty = a_0 < a_1 <a_2 < \cdots < a_n$ be the values of $f$. The (unextended) persistent homology concerns the persistence modules (one for each $p$)
\begin{equation*}
    0 = H_p(K_0) \rightarrow H_p(K_1) \rightarrow \cdots \rightarrow H_p(K_n)
\end{equation*}
where $K_i$ are the sublevel subcomplexes. 

Consider now the superlevel subcomplexes $K^j = \{ \sigma \in K | f(v) \geq a_i \mbox{ for all } v \in \sigma \}$. The \textit{extended persistent homology of $K$ with respect to $f$} is defined as the persistent module
\begin{equation}\label{eq:4}
  H_p(K_0) \rightarrow \cdots \rightarrow H_p(K_n) = H_p(K) \rightarrow H_p(K, K^n) \rightarrow \cdots \rightarrow H_p(K, K^0) = 0
\end{equation}

Intuitively, for unextended PH, the complex $K$ is scanned from bottom to the top with respect to $f$, while in the extended version the scanner turns downward after reaching the top to harvest more information.

The barcode of this module consists of three types of intervals:

(i) those ending before or at the term $H_p(K)$

(ii) those starting after the term $H_p(K)$

(iii) those starting before or at the term $H_p(K)$, and ending after it\\
The intervals of type (i), (ii), (iii) are called \textit{ordinary}, \textit{relative}, and \textit{extended} respectively. The extended barcode contains strictly more information than the unextended one: the latter corresponds to ordinary intervals and the starting points of extended intervals.

The author of \cite{CohenSteiner2009ExtendingPH} proposed a method of computing extended PH. It relies on the notion of simplicial cones. Let $K$ be a simplicial complex, define the \textit{cone} of $K$ as the simplicial complex $C(K) = K \cup \{ \sigma \cup \{\ast\} | \sigma \in K \}$ where $\ast$ is a vertex outside $K$. Since the geometric realization of $C(K)$ is the cone $|K| \times I / |K| \times \{0\}$ (where $| \cdot |$ stands for geometric realization), we have $H_p(K, K^j) = \widetilde{H}_p (K / K^j) = \widetilde{H}_p(K \cup C(K^j))$, where $\widetilde{H}$ stands for reduced homology groups. Thus if all homology groups in (\ref{eq:4}) are replaced with the reduced versions, the resulted persistence module would be isomorphic to
\begin{equation*}
  \widetilde{H}_p(K_0) \rightarrow \cdots \rightarrow \widetilde{H}_p(K_n) = \widetilde{H}_p(K) \rightarrow \widetilde{H}_p(K \cup C(K^n)) \rightarrow \cdots \rightarrow \widetilde{H}_p(K \cup C(K^0))
\end{equation*}

This is the persistent homology of a filtration on the simplicial complex $C(K)$, and thus can be decomposed using (the reduced version of) standard PH algorithm.

\begin{remark}
 The function $f$ is needed only for applications in data analysis. In fact, the above discussion can be used to define and compute extended persistent homology for any two filtrations $\{K_i\}$ and $\{K^j\}$ on a simplicial complex $K$.
\end{remark}

\subsection{An Algebraic Reformulation}
The above treatment of extended PH relies on the simplicial setting for the usage of cones. This has at least two undesirable consequences. First, it forces one to use reduced homology, which hides the first component born with respect to the filtration. In most applications, all components should be treated in the same way, so this concealment is unwanted. Secondly, the simplicial assumption hinders generalizations, e.g., to persistent path homology or persistent hypergraph homology.
In this subsection, we give an algebraic reformulation of extended persistent homology that generalizes the standard extended PH. This generalization provides a framework for discussing extended PH of (filtration of) graded subgroups.

Let $\{C_\ast, \partial_\ast\}$ be a chain complex. Let
\begin{align*}
    S^1_\ast \hookrightarrow \cdots \hookrightarrow S^i_\ast \hookrightarrow S^{i+1}_\ast \hookrightarrow \cdots \hookrightarrow S^{M}_\ast \\
    T^1_\ast \hookrightarrow \cdots \hookrightarrow T^j_\ast \hookrightarrow T^{j+1}_\ast \hookrightarrow \cdots \hookrightarrow T^{N}_\ast
\end{align*}
be two filtrations of subcomplexes in $C_\ast$ such that $S^{M}_\ast = T^N_\ast$ (in general, $S^i_\ast$ and $T^i_\ast$ are not required to contain one another). Define the \textit{extended persistent homology} of $(\{S_\ast^i\}_{1 \leq i \leq M}, \{T_\ast^j\}_{1 \leq j \leq N})$ as the persistence modules
\begin{equation}\label{AlgebraicExtendedPH}
    H_p(S^1_\ast) \rightarrow \cdots \rightarrow H_p(S^M_\ast) \rightarrow H_p(S^M_\ast / T^1_\ast) \rightarrow \cdots \rightarrow H_p(S^M_\ast / T^N_\ast) = 0
\end{equation}

In order to compute the decomposition of (\ref{AlgebraicExtendedPH}), we need the notion of mapping cones of chain complexes. Let $\{C_\ast, \partial_\ast\}$, $\{C'_\ast, \partial'_\ast\}$ be chain complexes and $\varphi : C'_\ast \rightarrow C_\ast$ be a chain map. The \textit{mapping cone} of $\varphi$ is the chain complex $\mbox{Cone}(\varphi) = \{\overline{C}_\ast, \overline{\partial}_\ast\}$ where $\overline{C}_p = C'_{p-1} \oplus C_p$ and 
\begin{equation*}
\overline{\partial}_p(c',c) = (-\partial_{p-1}' (c'), \varphi(c') + \partial_p(c)) \in C'_{p-2} \oplus C_{p-1} \mbox{  for  } c \in C'_{p-1}, c \in C_p
\end{equation*}

As is the case with topological cones, there is an embedding $C_\ast \hookrightarrow \overline{C}_\ast$ taking $c \in C_p$ to $(0,c) \in \overline{C}_p$.
We are interested in the case where $\varphi$ is the inclusion of a subcomplex. When there is no ambiguity over $C_\ast$, we will denote $\mbox{Cone}(\varphi)$ by $\mbox{Cone}(C'_\ast)$. The following is an algebraic analogy of the contractibility of cones.
\begin{lemma}
Let $C'_\ast$ be a subcomplex of $C_\ast$. Denote $\mbox{Cone}(C'_\ast)$ by $\{\overline{C}_\ast, \overline{\partial}_\ast\}$. Then the map 
\begin{equation*}
\begin{matrix}
    h:H_p(C_\ast / C'_\ast) &\Longrightarrow & H_p(\overline{C}_\ast) \\
     [\Bar{x}]& \Longrightarrow &[(-\partial x, x)]
\end{matrix}
\end{equation*}
is an isomorphism for all $p$, where $\Bar{x} \in C_p / C'_p$ is represented by $x \in C_p$ and $[ \cdot ]$ means taking homology class.
\end{lemma}
\begin{proof}

We begin with checking this is well-defined. Let $\partial_0$ denote the boundary map of $C_\ast / C'_\ast$.
For $\Bar{x} \in \mbox{Ker}(\partial_0)$, we have $\partial x \in C'_{p-1}$, hence $(-\partial x, x) \in \overline{C}_p$. Since $\partial (-\partial x, x) = (\partial (\partial x), -\partial x +\partial x) = 0$, the element $[(-\partial x, x)]$ is well-defined. It remains to check the independence of choice of $x$. Let $x_1,x_2 \in C_p$ such that $[\overline{x_1}] = [\overline{x_2}]$. Then $x_1-x_2 =\partial y + z$, for $y \in C_{p+1}, z \in C'_p$. Apply $\partial$, we get $\partial x_1 - \partial x_2 =\partial z$. Hence
\begin{equation*}
    \begin{split}
        (-\partial x_1, x_1) - (-\partial x_2, x_2)  & = (-\partial (x_1 -x_2), x_1 -x_2) \\
                                                     & = (-\partial z , \partial y +z ) \\
                                                     & = \partial (z, y)
    \end{split}
\end{equation*}
and $[(-\partial x_1, x_1)] =[(-\partial x_2, x_2)]$.

We now construct the inverse of $h$. Let $[(y,x)]$ be an element of $H_p(\overline{C}_\ast)$. We have $\partial(y,x) = (-\partial y, y + \partial x) =0$, thus $\partial x = -y \in C'_{p-1}$ and $[\Bar{x}] \in H_p(C_\ast / C'_\ast)$. If $[(-\partial x_1,x_1)] = [(-\partial x_2,x_2)] \in H_p(\overline{C}_\ast)$, then $(-\partial x_1, x_1) - (-\partial x_2,x_2) = (-\partial (x_1 - x_2),x_1 - x_2) = \partial (u,v)$ for some $u \in C'_p, v \in C_{p+1}$. This implies $x_1 - x_2 = u + \partial v$, and thus $[\overline{x_1}]=[\overline{x_2}]$. Hence the map 
\begin{equation*}
\begin{matrix}
    H_p(\overline{C}_\ast) &\Longrightarrow & H_p(C_\ast / C'_\ast) \\
     [(y, x)]& \Longrightarrow & [\Bar{x}]
\end{matrix}
\end{equation*}
is well-defined. It is easy to check that this is inverse to $h$.
\end{proof}

From the construction in the above proof, we see the isomorphism $h$ is natural with respect to $C'_\ast$: if $C^1_\ast \subset C^2_\ast$ are two subcomplexes of $C_\ast$, then $\overline{C}^1_\ast = \mbox{Cone}(C^1_\ast)$ is a subcomplex of $\overline{C}^2_\ast=\mbox{Cone}(C^2_\ast)$
and the diagram

\begin{tikzcd}
 H_\ast (C_\ast / C^1_\ast) \arrow[r, "h"] \arrow[d] & H_\ast (\overline{C}^1_\ast) \arrow[d]\\
 H_\ast (C_\ast / C^2_\ast) \arrow[r, "h"]         & H_\ast (\overline{C}^2_\ast)
\end{tikzcd} \\
commutes, where the left vertical map is induced by natural projection while the right vertical map by inclusion.
This implies:
\begin{corollary} \label{Corollary1}
Suppose $\{S_\ast^i\}_{1 \leq i \leq M}, \{T_\ast^j\}_{1 \leq j \leq N}$ are two filtrations of subcomplexes of $C_\ast$ such that $S^M_\ast = T^N_\ast$. Denote by $\overline{T}^j_\ast$ the mapping cone of $T^j_\ast \hookrightarrow S^M_\ast$. Then $S^1_\ast \subset \cdots \subset S^M_\ast \subset \overline{T}^1_\ast \subset \cdots \subset \overline{T^N_\ast}$ is a filtration on $\overline{T}^N_\ast$, and the persistence module (\ref{AlgebraicExtendedPH}) is isomorphic to 
\begin{equation} \label{ConeSubcomplexes}
    H_p(S^1_\ast) \rightarrow \cdots \rightarrow H_p(S^M_\ast) \rightarrow H_p(\overline{T}^1_\ast) \rightarrow \cdots \rightarrow H_p(\overline{T}^N_\ast) = 0
\end{equation} \qed
\end{corollary}

If we are given compatible basis for both $\{S_\ast^i\}_{1 \leq i \leq M}$ and $\{T_\ast^j\}_{1 \leq j \leq N}$, we can write down a compatible basis for $S^1_\ast \subset \cdots \subset S^M_\ast \subset \overline{T}^1_\ast \subset \cdots \subset \overline{T}^N_\ast$ and compute the interval decomposition of (\ref{ConeSubcomplexes}) by the standard PH algorithm (Section \ref{StandardPHAlgorithm}). Again, we postpone the detail here to the next subsection, where we study a generalized version.

\subsection{Extended PH for Graded Subgroups }
$\label{extended PH for Graded Subgroups} $\label{Extended PH for Graded Subgroups}
We are interested in generalizing extended PH to filtrations of graded subgroups. The first step is to define relative homology in this setting.

Let $E_\ast \subset D_\ast$ be graded subgroups of a chain complex $C_\ast$. Define the \textit{relative supremum homology group} as 
\begin{equation*}
    H_p^{sup}(D_\ast, E_\ast)=H_p^{sup}(D_\ast, E_\ast;C_\ast) \overset{\Delta}{=}H_p(S_\ast, T_\ast) = H_p(S_\ast / T_\ast)
\end{equation*}
 where $S_\ast, T_\ast$ denote the supremum complexes of $D_\ast, E_\ast$ in $C_\ast$ respectively. The relative infimum homology group $H_p^{inf}(D_\ast, E_\ast;C_\ast)$ is defined analogously. An easy argument by the Five Lemma shows that $H_p^{sup}(D_\ast, E_\ast)$ is naturally isomorphic to $H_p^{inf}(D_\ast, E_\ast)$. By Remark \ref{Graded Subgroups of a Chain Complex}, the relative homology groups are not affected by replacing $C_\ast$ with a larger chain complex. Naturality with respect to inclusion is also available: if $D^i_\ast,E^i_\ast,i=1,2$ are graded subgroups of $C_\ast$ and $D^1_\ast \subset D^2_\ast, E^1_\ast \subset E^2_\ast,  E^i_\ast \subset D^i_\ast,i=1,2$, then there is a canonical homomorphism $H^{sup}_p(D^1,E^1) \rightarrow H^{sup}_p(D^2,E^2)$ induced by inclusion.

To compute relative homology, we need the notion of cones for graded subgroups. Denote the mapping cone $\mbox{Cone}(id: C_\ast \rightarrow C_\ast)$ by $\overline{C}_\ast$. Define the \textit{mapping cone} of $E_\ast \hookrightarrow D_\ast$ as the graded subgroup $\mbox{Cone}(E_\ast \hookrightarrow D_\ast) = \overline{E}_\ast$ of $\overline{C}_\ast$, where $\overline{E}_p = E_{p-1} \oplus D_p \subseteq \overline{C}_p$ for each $p$. The following proposition shows that the cone construction commutes with taking supremum subcomplex.

\begin{prop} \label{Commute}
In the above setting, let $S_\ast$ (resp. $T_\ast$) be the supremum subcomplex of $D_\ast$ (resp. $E_\ast$) and set $\overline{T}_\ast = \mbox{Cone}(T_\ast \hookrightarrow S_\ast)$. Then with respect to the natural embedding $\overline{T}_\ast \hookrightarrow \overline{C}_\ast$, $\overline{T}_\ast$ is the supremum subcomplex of $\overline{E}_\ast = \mbox{Cone}(E_\ast \hookrightarrow D_\ast)$ in $\overline{C}_\ast$.
\end{prop}
\begin{proof}
We need to check that $\overline{T}_p$ coincide with $\overline{E}_p + \partial \overline{E}_{p+1}$. By definition,
\begin{equation*}
    \overline{T}_p = T_{p-1} \oplus S_p = (E_{p-1} + \partial E_p) \oplus (D_p + \partial D_{p+1}) \\
\end{equation*}
while
\begin{equation*}
    \overline{E}_p + \partial \overline{E}_{p+1} = E_{p-1} \oplus D_p + \partial (E_p \oplus D_{p+1})
\end{equation*}
It is thus obvious that $\overline{E}_p + \partial \overline{E}_{p+1} \subset \overline{T}_p$. Conversely, take $(x+\partial y, z+\partial w) \in \overline{T}_p$, where $x \in E_{p-1}, y \in E_p, z \in D_p, w \in D_{p+1}$. We compute
\begin{equation*}
    \begin{split}
       (x, y+z) + \partial(-y, w)                    & = (x, y+z) +(\partial y, -y + w)\\
                                                     & = (x+\partial y, z+\partial w)
    \end{split}
\end{equation*}

Since $E_p \subset D_p$, we have $(x, y+z) + \partial(-y, w) \in  \overline{E}_p + \partial \overline{E}_{p+1}$.
\end{proof}

Suppose now are given a chain complex $C_\ast$ and two filtrations of graded subgroups
\begin{align*}
    D^1_\ast \hookrightarrow \cdots \hookrightarrow D^i_\ast \hookrightarrow D^{i+1}_\ast \hookrightarrow \cdots \hookrightarrow D^{M}_\ast \\
    E^1_\ast \hookrightarrow \cdots \hookrightarrow E^j_\ast \hookrightarrow E^{j+1}_\ast \hookrightarrow \cdots \hookrightarrow E^{N}_\ast
\end{align*}
with $D^{M}_\ast = E^{N}_\ast$. As is expected, the $p$-dimensional extended persisitent homology of $(\{D_\ast^i\}_{1 \leq i \leq M}, \{E_\ast^j\}_{1 \leq j \leq N})$ is defined as the persistence module
\begin{equation} \label{ExtendedPHGradedSubgroupsDef}
    H_p^{sup}(D^1_\ast) \rightarrow \cdots \rightarrow H_p^{sup}(D^M_\ast) \rightarrow H_p^{sup}(D^M_\ast, E^1_\ast) \rightarrow \cdots \rightarrow H_p^{sup}(D^M_\ast, E^N_\ast) = 0
\end{equation}
where all maps are induced by inclusion.

In other words, if $S^i_\ast$ and $T^j_\ast$ ($1 \leq i \leq M, 1 \leq j \leq N$) denote the supremum complexes of $D^i_\ast$ and $E^j_\ast$ respectively, then (\ref{ExtendedPHGradedSubgroupsDef}) is just the extended PH of $(\{S_\ast^i\}_{1 \leq i \leq M}, \{T_\ast^j\}_{1 \leq j \leq N})$.
Note that if $H^{\inf}$ is used in (\ref{ExtendedPHGradedSubgroupsDef}) instead of $H^{\sup}$, we would get an persistence module isomorphic to (\ref{ExtendedPHGradedSubgroupsDef}).

Our goal is to compute the interval decomposition of (\ref{ExtendedPHGradedSubgroupsDef}). Define $\overline{E}^j_\ast = \mbox{Cone}(E^j_\ast \hookrightarrow D^M_\ast) \subset \overline{C}_\ast = \mbox{Cone}(id:C_\ast \rightarrow C_\ast)$ and identify $D_\ast^i$ with $0 \oplus D^i_\ast \subset \overline{C}_\ast$. Then
\begin{equation}\label{FiltrationGradedSubgroups}
    D^1_\ast \subset \cdots \subset D^M_\ast \subset \overline{E}^1_\ast \subset \cdots \subset \overline{E}^N_\ast
\end{equation}
is a filtration of graded subgroups of $\overline{C}_\ast$. By Proposition \ref{Commute}, taking supremum complexes in (\ref{FiltrationGradedSubgroups}) gives
\begin{equation*}
    S^1_\ast \subset \cdots \subset S^M_\ast \subset \overline{T}^1_\ast \subset \cdots \subset \overline{T}^N_\ast
\end{equation*}

By definition, the persistence module (\ref{ExtendedPHGradedSubgroupsDef}) is the same as (\ref{AlgebraicExtendedPH}), which is in turn isomorphic (Corollary \ref{Corollary1}) to (\ref{ConeSubcomplexes}). The above discussion shows that (\ref{ConeSubcomplexes}) is the persistent homology of the filtration (\ref{FiltrationGradedSubgroups}), so we can apply the method of Section \ref{ComputePHbySup}.

\subsection{Algorithm for Extended PH}
We are now ready to present an algorithm for computing extended PH of graded subgroups. The inputs are filtrations
\begin{align*}
    D^1_\ast \hookrightarrow \cdots \hookrightarrow D^i_\ast \hookrightarrow D^{i+1}_\ast \hookrightarrow \cdots \hookrightarrow D^{M}_\ast \\
    E^1_\ast \hookrightarrow \cdots \hookrightarrow E^j_\ast \hookrightarrow E^{j+1}_\ast \hookrightarrow \cdots \hookrightarrow E^{N}_\ast
\end{align*}
of graded subgroups of $C_\ast$ with $D^{M}_\ast = E^{N}_\ast$, and compatible basis $\{d_p^1, \cdots, d_p^{m_p}\}$, $\{e_p^1, \cdots, e_p^{m_p}\}$ of $\{D^i_\ast\}_{1\leq i \leq M}$, $\{E^j_\ast\}_{1\leq j \leq N}$ respectively.  Then
\begin{equation*}
\{(0,d_p^1), \cdots ,(0,d_p^{m_p}), (e_{p-1}^1,0), \cdots, (e_{p-1}^{m_{p-1}},0)\}    
\end{equation*}
is a compatible basis for (\ref{FiltrationGradedSubgroups}) at dimension $p$. In practice, $\{e_p^1, \cdots, e_p^{m_p}\}$ is usually a permutation away from $\{d_p^1, \cdots, d_p^{m_p}\}$.

Since $D^{M}_\ast = E^{N}_\ast$, for each $p$ there is a common extension of basis, that is, a set $B_p$ such that both $\{d_p^1, \cdots, d_p^{m_p}\}\cup B_p$ and $\{e_p^1, \cdots, e_p^{m_p}\} \cup B_p$ are basis of $C_p$. Denote by $\varepsilon_p^1, \cdots, \varepsilon_p^{l_p}$ the elements of $B_p$ that appears in the boudanry of $d_{p+1}^i$'s or $e_{p+1}^j$'s. Let $\overline{\partial}_\ast$ be the boundary map of $\overline{C}_\ast$, and define $A_{p+1}$ as the matrix of $\overline{\partial}_{p+1}$ with respect to 
\begin{equation*}
    (0,d_{p+1}^1), \cdots ,(0,d_{p+1}^{m_{p+1}}), (e_{p}^1,0), \cdots, (e_{p}^{n_{p}},0)
\end{equation*}
in the domain and
\begin{multline*}
    (0,d_p^1), \cdots ,(0,d_p^{m_p}), (e_{p-1}^1,0), \cdots, (e_{p-1}^{n_{p-1}},0),(0,\varepsilon^1_p), \cdots, (0,\varepsilon^{l_p}_p),\\
     (\varepsilon^1_{p-1}, 0), \cdots, (\varepsilon^{l_{p-1}}_{p-1}, 0)
\end{multline*}
in the codomain.

Reduce $A_{p+1}$ by left-to-right additions of columns and record the resulted pairing. Let $ht, ht'$ denote the height function with respect to $\{D^i_\ast \}$ and $\{E^j_\ast \}$ respectively. Then we have:
\begin{theorem}
The interval decomposition of \ref{ExtendedPHGradedSubgroupsDef} is determined by the pairings in the following way:

(i) each pair $((0,d_p^i),(0,d_{p+1}^j))$ with $ht(d_p^i) < ht(d_{p+1}^i)$ contributes an interval spanning from $H_p^{sup}(D_\ast^{ht(d_p^i)})$ to $H_p^{sup}(D_\ast^{ht(d_{p+1}^j)-1})$.

(ii) each pair $((e_p^i,0),(e_{p+1}^j,0))$ with $ht'(e_p^i) < ht'(e_{p+1}^j)$ contributes an interval spanning from $H_p^{sup}(E_\ast^{ht'(e_p^i)})$ to $H_p^{sup}(E_\ast^{ht'(e_{p+1}^j)-1})$.

(iii) each pair $((e_p^i,0),(0,d_{p+1}^j))$ contributes an interval from $H_p^{sup}(D_\ast^{ht(d_p^i)})$ to $H_p^{sup}(E_\ast^{ht'(e_{p+1}^j)-1})$. \qed

\end{theorem}

\section{Stability}
The application of a mathematical tool in data analysis typically demands its stability. That is, small perturbation in input data should result in small change in the output. The stability of unextended persistent (simplicial) homology (c.f. \cite{CT} p.182) has been documented ever since the early days of the theory of PH. Stability is often formulated as an inequality stating that the bottleneck distance is not greater than the interleaving distance. The stability for persistent path homology and hypergraph homology, our principal examples, are documented in \cite{Chowdhury18} and \cite{ren2020stability} respectively. We intend to generalize these to the case of extended persistence. In \cite{chazal2016structure} (Section 6.2), the stability for extended persistent simplicial homology is briefly discussed. We will present a stability theorem (Theorem \ref{Stability}) for extended persistence modules (defined later) in general, and apply this theorem to the extended persistence of path homology and hypergraph homology.

In Section \ref{Continuous Extended Persistence Modules}, we present the basic notions necessary for formulating the stability results. In Section \ref{Bottleneck Distance and Interleaving}, we discuss the bottleneck distance and interleaving distance in the extended setting, which are measurements for small changes and prove the stability theorem (Theorem \ref{Stability}). In the remaining two subsections, we state and proof stability theorems for path homology and hypergraph homology respectively.

\subsection{Continuous Extended Persistence Modules} \label{Continuous Extended Persistence Modules}

 Define $\mathbb{R}^o$ as the poset $\mathbb{R}$ with reversed order. For $b \in \mathbb{R}$, denote the corresponding element of $\mathbb{R}^o$ as $\Bar{b}$. Let $E = \mathbb{R} \cup \{\infty\} \cup \mathbb{R}^o$ with the ordering $s < \infty < \Bar{t}$ for all $s,t \in \mathbb{R}$. A \textit{(continuous) extended persistence module} $V$ is a functor from the poset $E$ (regarded as a category) to the category $\mbox{vect}_k$ of finite dimensional vector spaces over $k$. In other words, each $x \in E$ is assigned a vector space $V_x$ and each pair $x < y$ a linear map $V_{x,y}: V_x \rightarrow V_y$ such that

(i) $V_{x,x} = id$

(ii) $V_{x,z} = V_{x,y} \circ V_{y,z}$\\
The definition of morphism, direct sum and interval module are analogous to the discrete case. Note that here intervals are those with respect to the poset $E$ (e.g., $(3,\Bar{2}], (\infty, \Bar{1})$).

Here is how extended persistence modules arise from discrete ones defined earlier in this paper.
\begin{exmp} \label{Example}
Given the persistence module (\ref{ExtendedPHGradedSubgroupsDef}) and sequences of real numbers $a_1 < \cdots < a_M, b_1 > \cdots > b_N$, we can define a extended persistence module $V$ by
\begin{equation*}
    V_x=
    \begin{cases}
    0 & x < a_1 \\
    H_p(D^i) & a_i \leq x < a_{i+1}, 1 \leq i\leq M-1\\
    H_p(D^M) & a_M \leq x < \overline{b_1} \\
    H_p(D^M, E^j) & \overline{b_j} \leq x < \overline{B_{j+1}}, 1 \leq j \leq N-1 \\
    H_p(D^M, E^N)=0 & x \geq \overline{b_N}
    \end{cases}
\end{equation*}
where $H$ means $H^{sup}$ and $V_{x,y}$'s are induced by inclusion.
\end{exmp}

For our purpose, it suffices to deal with a smaller class of persistence modules. A extended persistence module $V$ is called \textit{decomposable} if it is isomorphic to a direct sum of finitely many interval modules over $E$, none of which has $\infty$ as an endpoint. Such decomposition is always unique (see \cite{chazal2016structure} Theorem 2.7, note that $E$ is isomorphic to $\mathbb{R}$ as a poset). Decomposability can be alternatively formulated in the following way:
\begin{prop} \label{decomposability}
An extended persistence module $V$ is decomposable if and only if:

(i) $V_t$ is finite dimensional for all $t \in E$

(ii) there exist $a, b\in \mathbb{R}$, such that $V_{s,t}$ is an isomorphism for all $a<s<t<\overline{b}$ (i.e., $V$ is \textit{locally constant} near $\infty$)
\end{prop}
\begin{proof}
For the "if" part, condition (i) implies the existence of interval decomposition (c.f. \cite{chazal2016structure} Theorem 2.8). Local constancy near $\infty$ guarantees that none of the intervals has $\infty$ as endpoint. The "only if" part is trivial.
\end{proof}
 As an example, the module in Example \ref{Example} is decomposable if all homology groups involved are finite dimensional.
\par
The definition of persistence diagrams involves multisets. For our purposes, a \textit{multiset} is a
pair $A = (S, m)$ where $S$ is a set and
$$m:S\to \{1, 2,...\}\cup \infty$$
is the multiplicity function, which counts the occurrence of an element in A. The \textit{cardinality} $card(A)$ of $A=(S,m)$ is defined to be $\sum_{s\in S}m(s)$ if this sum is well-defined and finite, and $\infty$ otherwise.
We now move to define the \textit{extended persistence diagrams} of a decomposable module $V$. For simplicity, we denote an interval in $E$ of the form $(a,b),(a,b],[a,b),[a,b]$ as $(a^+,b^-),(a^+,b^+),(a^-,b^-),(a^-,b^+)$ respectively. There are then 3 types of intervals in the decomposition:

(i) ordinary: $(a^\pm,a'^\pm)$

(ii) relative: $(\Bar{b}^\pm,\Bar{b'}^\pm)$

(iii) extended: $(a^\pm,\Bar{b}^\pm)$.

The \textit{ordinary persistence diagram} of $V$ is obtained by placing a point $(a,a')$  on the plane $\mathbb{R} \times \mathbb{R}$ for each ordinary interval $(a^\pm,a'^\pm)$, counting multiplicity. The \textit{relative} (resp. \textit{extended}) \textit{persistence diagram} is defined similarly, except it lies on the plane $\mathbb{R} \times \mathbb{R}^o$ (resp. $\mathbb{R}^o \times \mathbb{R}^o$). The 3 types of persistence diagrams are denoted by $\mbox{Ord}(V)$, $\mbox{Rel}(V)$ and $\mbox{Ext}(V)$ respectively. Note that they are multisets on the respective planes.

\subsection{Bottleneck Distance and Interleaving}
\label{Bottleneck Distance and Interleaving}

There are two notions of distance between persistence modules: bottleneck distance and interleaving distance. Roughly speaking, the bottleneck distance measures how close the persistence modules are by comparing their persistence diagrams, while the interleaving distance measures how far away they are from being isomorphic and typically can be related to the input data. In the unextended case, the two distances are equated by the Isometry Theorem (see \cite{chazal2016structure} Theorem 5.14) if certain finiteness assumptions are satisfied.

We now formulate these notions in the extended setting, starting with bottleneck distance. Throughout this discussion, we will be using the $l^\infty$ distance on the plane:
\begin{equation*}
    d((x,y),(z,w)) = \mbox{max} (|x-z|, |y-w|)
\end{equation*}

Given two multisets $A,B$, a \textit{partial matching} is a bijection $\phi: A' \longleftrightarrow B'$ between subsets $A' \subset A, B' \subset B$. If $A'=A, B'=B$, we call $\phi$ a \textit{perfect matching}. 

Let $V,W$ be decomposable persistence modules and $\delta >0$. A $\delta$-matching between $V$ and $W$ is a triple $\Phi = \{\phi_O, \phi_R, \phi_E\}$ where $\phi_O$ (resp. $\phi_R$) is a partial matching between $\mbox{Ord}(V)$ and $\mbox{Ord}(W)$ (resp. $\mbox{Rel}(V)$ and $\mbox{Rel}(W)$) while $\phi_E$ is a perfect matching between $\mbox{Ext}(V)$ and $\mbox{Ext}(W)$ such that all matched pairs are $\delta$-close and unmatched points are $\delta$-close to the diagonal in the plane. The \textit{bottleneck distance} between $V,W$ is defined by
\begin{equation*}
    d_B(V,W) = \mbox{inf} \{\delta \mbox{ }|\mbox{ } V, W \mbox{ are } \delta \mbox{-matched}\}
\end{equation*}
In particular, the requirement of perfect matching for the extended diagram means that $d_B(V,W)= +\infty$ unless $\mbox{dim}V_\infty =\mbox{dim}W_\infty$.

We now turn to interleaving. 
\begin{definition} \label{interle}

 Let $V,W$ be extended persistence modules and $\varepsilon$ be a positive real number. A $\varepsilon$-\textit{interleaving} between $V$ and $W$ is a quadruple of families of maps $(\{\varphi_a\}, \{\varphi_{\Bar{b}}\},\{\psi_a\}, \{\psi_{\Bar{b}}\})$ where
\begin{equation*}
     \begin{matrix}
    \varphi_a: V_a \longrightarrow W_{a+\varepsilon} & \varphi_{\Bar{b}}: V_{\overline{b}} \longrightarrow W_{\overline{b -\varepsilon}}\\
    \psi_a : W_a \longrightarrow V_{a+\varepsilon} & \psi_{\Bar{b}}: W_{\overline{b}} \longrightarrow V_{\overline{b - \varepsilon}}
    \end{matrix}
\end{equation*}
for $a \in \mathbb{R}, \overline{b} \in \mathbb{R}^o$ satisfying the following conditions:

(i) For $a < a' \in \mathbb{R}, \overline{b} < \overline{b'} \in \mathbb{R}^o$, we have the following naturality conditions:
\begin{equation*}
    \begin{matrix}
        W_{a+\varepsilon, a'+\varepsilon} \circ \varphi_a = \varphi_{a'} \circ V_{a,a'} \\
        W_{a+\varepsilon, \overline{b-\varepsilon}} \circ \varphi_a = \varphi_{\overline{b}} \circ V_{a,\overline{b}} \\
        W_{\overline{b-\varepsilon}, \overline{b'-\varepsilon}} \circ \varphi_{\overline{b}} = \varphi_{\overline{b'}} \circ V_{\overline{b},\overline{b'}} 
    \end{matrix}
\end{equation*}
In other words, the following diagrams commute:

\begin{tikzcd}
   & W_{a+\varepsilon} \arrow[rr, "W_{a+\varepsilon, a'+\varepsilon}"] & & W_{a'+\varepsilon} \\
   V_a \arrow[ur, "\varphi_a"] \arrow[rr, "V_{a,a'}"]& & V_{a'} \arrow[ur, "\varphi_{a'}"] \\
   & W_{a+\varepsilon} \arrow[rr, "W_{a+\varepsilon, \overline{b-\varepsilon}}"] & & W_{\overline{b-\varepsilon}} \\
   V_a \arrow[ur, "\varphi_a"] \arrow[rr, "V_{a,\overline{b}}"]& & V_{\overline{b}} \arrow[ur, "\varphi_{\overline{b}}"] \\
   & W_{\overline{b-\varepsilon}} \arrow[rr, "W_{\overline{b-\varepsilon}, \overline{b'-\varepsilon}}"] & & W_{\overline{b'-\varepsilon}} \\
   V_{\overline{b}} \arrow[ur, "\varphi_a"] \arrow[rr, "V_{\overline{b},\overline{b'}}"]& & V_{\overline{b'}} \arrow[ur, "\varphi_{\overline{b'}}"]
\end{tikzcd}

(i') The 3 equations in (i) remains true with $V,W$ exchanged and $\varphi$ replaced by $\psi$.

(ii) For $a \in \mathbb{R}, \overline{b} \in \mathbb{R}^o$, the following equations hold:
\begin{equation*}
    \begin{matrix}
        \psi_{a+\varepsilon} \circ \varphi_a = V_{a,a+\varepsilon} \\
        \psi_{\overline{b-\varepsilon}} \circ \varphi_{\overline{b}} =V_{\overline{b},\overline{b-2\varepsilon}}
    \end{matrix}
\end{equation*}
These amounts to 
commutativity of diagrams:

\begin{tikzcd}
    & W_{a+\varepsilon} \arrow[dr, "\psi_{a+\varepsilon}"] & \\
    V_a \arrow[rr, "V_{a,a+2\varepsilon}"] \arrow[ur, "\varphi_a"]  & & V_{a+2\varepsilon} \\
    & W_{\overline{b- \varepsilon}} \arrow[dr, "\psi_{\overline{b-\varepsilon}}"] & \\
    V_{\overline{b}} \arrow[rr, "V_{\overline{b},\overline{b-2\varepsilon}}"] \arrow[ur, "\varphi_{\overline{b}}"]  & & V_{\overline{b-2\varepsilon}}
\end{tikzcd}

(ii') The 2 equations in (ii) remains true with the role of $(V,\varphi)$ and $(W,\psi)$ reversed.\\
\end{definition}

The \textit{interleaving distance} between $V,W$ is defined by:
\begin{equation*}
    d_I(V,W) = \mbox{inf} \{\varepsilon \mbox{ }|\mbox{ } V, W \mbox{ are } \varepsilon \mbox{-interleaved}\}
\end{equation*}

We shall prove that the assumptions of decomposability are sufficient for stability, i.e., the bottleneck distance of two extended persistence modules is bounded above by their interleaving distance. 

The proof of our stability theorem makes use of rectangle measures. For our purpose, we define an \textit{admissible rectangle} as a planar rectangle of one of the following forms:
$$[a, b]\times [c, d] \subseteq \mathbb{R}\times \mathbb{R}, a <b<  c<d\in \mathbb{R}$$
 $$[\overline{a}, \overline{b}]\times [\overline{c}, \overline{d}] \subseteq \mathbb{R}^o  \times \mathbb{R}^o, a >b>  c>d\in \mathbb{R}$$
 $$[a, b]\times [\overline{c}, \overline{d}] \subseteq \mathbb{R}\times \mathbb{R}^o,  a <b,  c>d\in \mathbb{R}$$
 
 We define three rectangle measures, one for each plane, that assign a nonnegative integer or $+\infty$ to a admissible rectangle $T$.
 $$u^O_V(T):=card (\mbox{Ord}(V)\mid_{T})$$
 $$u^R_V(T):=card (\mbox{Rel}(V)\mid_{T})$$
 $$u^E_V(T):=card (\mbox{Ext}(V)\mid_{T})$$
 
 Given $\delta>0$, the $\delta$-thickening of a rectangle is defined by
 $$([a, b]\times [c, d])^{\delta}=[a-\delta,b+\delta]\times [c-\delta,d+\delta]$$
 $$([\overline{a}, \overline{b}]\times [\overline{c}, \overline{d}])^\delta=[\overline{a+\delta},\overline{b-\delta}]\times [\overline{c+\delta},\overline{d-\delta}]$$$$([a,b]\times [\overline{c}, \overline{d}])^\delta=[a-\delta,b+\delta]\times [\overline{c+\delta},\overline{d-\delta}]$$

The following lemmas are the analogy of Theorem 5.26 and Theorem 5.29 of \cite{chazal2016structure} in the extended setting.
 \begin{lemma}\label{measure}
 Let $U, V $ be a $\delta$-interleaved pair of extended persistence modules.Let $T$ be an admissible rectangle in $\mathbb{R}\times \mathbb{R},\mathbb{R}  \times \mathbb{R}^o$ or $\mathbb{R}^o\times \mathbb{R}^o$ whose $\delta$-thickening $T^{\delta}$ is also admissible. Then $u^*_U(T)\leq u^*_V(T^{\delta})$, $u^*_V(T)\leq u^*_U(T^{\delta})$ .
 \end{lemma}
 \begin{proof}
The proof of Theorem 5.26 of \cite{chazal2016structure} carries verbatim.
 \end{proof}

\begin{lemma}\label{Stability}
 For $\varepsilon>0$, if there exists a family $\{V^t| t \in [0, \varepsilon] \}$ of decomposable extended persistence modules such that $V^s, V^t$ are $|s-t|$-interleaved for all $s,t \in [0, \varepsilon]$, then $d_B(V^0,V^\varepsilon) \leq \varepsilon$.
\end{lemma}

\begin{proof}
 For any $s,t \in [0,\varepsilon]$ and any rectangle $T$ such that $T^{|s-t|}$ is admissible, we have 
$u^*_{V^s}(T)\leq u^*_{V^t}(T^{|s-t|}), u^*_{V^t}(T)\leq u^*_{V^s}(T^{|s-t|})$ by Lemma \ref{measure}. It suffices to find an $\varepsilon$-matching $\Phi = \{\phi_O, \phi_R, \phi_E\}$ between $V^0$ and $V^\varepsilon$. The partial matching $\phi_O$ can be obtained by applying Theorem 5.29 of \cite{chazal2016structure} to $(u^O_{V^s} | s \in [0,\varepsilon])$ with $\mathcal{D}$ being the open half plane above the diagonal (not including infinity). The matching $\phi_R$ is obtained analogously, using the natural symmetry between $\mathbb{R} \times \mathbb{R}$ and $\mathbb{R} \times \mathbb{R}^o$.
To get $\phi_E$, we utilize the symmetry between $\mathbb{R} \times \mathbb{R}$ and $\mathbb{R}^o \times \mathbb{R}^o$ and apply the above-mentioned theorem to $(u^E_{V^s} | s \in [0,\varepsilon])$ with $\mathcal{D}$ being the entire plane (not including infinity). The resulted $\phi_E$ has to be perfect by our choice of $\mathcal{D}$.
\end{proof}
By Lemma \ref{Stability}, given two $\varepsilon$-interleaved extended persistence modules, it suffices to join them by a family of mutually interleaved modules (as in Lemma \ref{Stability}). This family can be constructed using a Kan extension argument.

Given $\varepsilon \geq 0$, we define a translation $\Omega_\varepsilon$ on the subset $\mathbb{R} \cup \mathbb{R}^o$ of $E$ as follows 
\begin{equation*}
    \Omega_\varepsilon(x)=
    \begin{cases}
    x+\varepsilon & x \in \mathbb{R}  \\
    
    x-\varepsilon & x\in \mathbb{R}^o
    
    \end{cases}
\end{equation*}
\par
We can make $E\times \{0, \varepsilon\}$ a poset by setting    $(x, a) \leq (y, b)$ if one of the following holds:
\par(1)  $\Omega_{|a-b|}(x)  \leq  y$ ,  $x\neq \infty, y \neq \infty $
\par(2)   $x \in \mathbb{R}, y = \infty$
\par(3) $x =\infty, y \in \mathbb{R}^o$.
\begin{theorem}\label{exinterle}

Two  extended persistence module $V$, $W$  are  $\varepsilon$-interleaved if and only if the following functor extension problem has a solution:
\par
\begin{tikzcd}
                                   & Vect                                               &                                                \\
E \arrow[r, "i_0"] \arrow[ru, "V"] & E\times \{0, \varepsilon\} \arrow[u, "L", dotted] & E \arrow[l, "i_\varepsilon"'] \arrow[lu, "W"']
\end{tikzcd}
\end{theorem} 
\begin{proof}
The functor $i_0$ is an identity  embedding from $E$ to $E\times \{0\}$ and functor $i_\varepsilon$ is an identity  embedding from $E$ to $E\times \{\varepsilon\}$. We denote the extension functor $L$ and denote the morphism from $L((x, a))$  to $ L((y, b))$ as $L_{(x, a) ,  (y, b)} $  for any $(x, a) $, $ (y, b) \in E\times \{0, \varepsilon\}$.  If $V$, $W$ is decomposable, then L is finite dimensional, and has local constancy near $\infty$ which means there exist $a \in \mathbb{R}$  $b\in \mathbb{R}$, such that $V_{s,t}$  $W_{s,t}$ is an  isomorphism for all $a<s<t<\overline{b}$.
\par
For the "if" part, the existence of the solution can make us take a quadruple of families of maps ( $\{L_{(a, 0) ,  (a+\varepsilon, \varepsilon) } \}  $, $\{L_{(\Bar{b}, 0) ,  (\Bar{b-\varepsilon}, \varepsilon) } \}  $ , $\{L_{(a, \varepsilon) ,  (a+\varepsilon, 0) } \}  $, $\{L_{(\Bar{b}, \varepsilon) ,  (\Bar{b-\varepsilon}, 0) } \}  $  ) as  a quadruple of families of maps $(\{\varphi_a\}, \{\varphi_{\Bar{b}}\},\{\psi_a\}, \{\psi_{\Bar{b}}\})$ in the Definition \ref{interle}.The  the functoriality of $L$ implies the commutativity of diagrams in the Definition \ref{interle}.
\par For the "only if" part, extended persistence module $V$, $W$ are  $\varepsilon$-interleaved then there exist a quadruple of families of maps ($\{\varphi_a\}$, $\{\varphi_{\Bar{b}}\},\{\psi_a\}$, $\{\psi_{\Bar{b}}\}$) in the Definition \ref{interle} that can be taken as ( $\{L_{(a, 0) ,  (a+\varepsilon, \varepsilon) } \}  $, $\{L_{(\Bar{b}, 0) ,  (\Bar{b-\varepsilon}, \varepsilon) } \}  $ , $\{L_{(a, \varepsilon) ,  (a+\varepsilon, 0) } \}  $, $\{L_{(\Bar{b}, \varepsilon) ,  (\Bar{b-\varepsilon}, 0) } \}  $  ). The commutativity of diagrams in the Definition \ref{interle}  implies the  functoriality of $L$ and the commutativity of diagram in the Theorem \ref{exinterle}.

\end{proof}

\begin{theorem}\label{Kan extension}
If two decomposable extended persistence modules $V$, $W$  are $\varepsilon$- interleaved, then there exists a family $\{V^t| t \in [0, \varepsilon] \}$ of decomposable extended persistence modules such that $V^s, V^t$ are $|s-t|$-interleaved for all $s,t \in [0, \varepsilon]$, $V^0=V$, $V^\varepsilon=W$. 
\end{theorem}
\begin{proof}
 We can make $ E\times[0,\varepsilon]$ a poset by setting $(x, a) \leq (y, b)$ if one of the following holds:
\par(1)  $x \leq y$, $a=b$ 
\par(2)   $\Omega_{|a-b|}(x)  \leq  y$, $ a\neq b$ ,  $x\neq \infty $ \par(3) $x=\infty $, $y \in \mathbb{R}^o$, $ a\neq b$.

A family of persistence module $V^t$ is found for the follwing functor extension problem has a solution
\par
\begin{tikzcd}
                                                           & Vect                                                    \\
E\times \{0, \varepsilon\} \arrow[r, "i"] \arrow[ru, "L"] & {{E\times[0, \varepsilon]}} \arrow[u, "H", dotted]
\end{tikzcd}
\par
Here the functor $i$ is the identity embedding. Since small category $E\times \{0, \varepsilon\}$  is a full subcategory of $E\times[0,\varepsilon]$ , and the category $Vect$ contains all colimits, the problem is solved by taking a left Kan extension $H$. For any $(e,b)\in E \times [0,\varepsilon] $, 
\[
H(e,b) :=\displaystyle\lim_{i(x, t) \rightarrow (e, b)}L((x, t))=  \begin{cases}
    L((e, b))& (e,b)\in E\times \{0, \varepsilon\}  \\
    L((\Omega^{-1}_b (e), 0)\bigoplus  L(\Omega^{-1}_{\varepsilon-b} (e), \varepsilon)) /G& others
    
    \end{cases}
    \]
    Where G is a subgroup generated by the following element :
    $$\{(L_{(\Omega^{-1}_\varepsilon (\Omega^{-1}_b (e)), \varepsilon),(\Omega^{-1}_b (e), 0) }(x),  -L_{(\Omega^{-1}_\varepsilon (\Omega^{-1}_b (e)), \varepsilon),(\Omega^{-1}_{\varepsilon-b} (e), \varepsilon) }(x)) \arrowvert x\in L_{(\Omega^{-1}_\varepsilon (\Omega^{-1}_b (e)), \varepsilon)}\}$$
    $$\{(L_{(\Omega^{-1}_\varepsilon (\Omega^{-1}_{\varepsilon -b} (e)), 0),(\Omega^{-1}_b (e), 0) }(y),  -L_{(\Omega^{-1}_\varepsilon (\Omega^{-1}_{\varepsilon -b} (e)), 0),(\Omega^{-1}_{\varepsilon-b} (e), \varepsilon) }(y)) \arrowvert y \in L_{(\Omega^{-1}_\varepsilon (\Omega^{-1}_{\varepsilon -b} (e)), 0)}\}$$
    
 We  denote the morphism from $H((x, a))$  to $ H((y, b))$ as $H_{(x, a) ,  (y, b)} $  for any $(x, a) $, $ (y, b) \in   E \times [0,\varepsilon]$.  So $V^t(e):=H((e,t))$ has finite dimension, and has local constancy near $\infty$ which means there exist $a \in \mathbb{R}$  $b\in \mathbb{R}$, such that $V^t_{l,h}:=H_{(l,t),(h,t)}$   is an  isomorphism for all $a<l<h<\overline{b}$  for the image of L  has finite dimension, and has local constancy near $\infty$. Then by Proposition \ref{decomposability}  $\{V^t| t \in [0, \varepsilon] \}$ is  a family of decomposable extended persistence modules.

\end{proof}

\begin{theorem}\label{final stablity}

If two decomposable extended persistence module $V$, $W$ are  $\varepsilon$ interleaved  then $d_B(V^0,V^\varepsilon) \leq \varepsilon$
\end{theorem}
Theorem \ref{final stablity} is a corollary of theorem \ref{Stability} and theorem \ref{Kan extension}

\subsection{Stability for Persistent Path Homology}
We shall need the following definition of relative path homology of digraphs:
Let $G=(X,E),G'=(X,E')$ be two graphs on the same vertex set $X$ such that $E'\subset E$, then we have $\mathcal{A}_\ast(G') \subset \mathcal{A}_\ast(G) \subset \mathcal{R}_\ast(X)$. Define the \textit{relative path homology group} of $(G,G')$ as $H_p(G,G')=H^{sup}_p(\mathcal{A}_\ast(G),\mathcal{A}_\ast(G');\mathcal{R}_\ast(X))$.

For digraphs, a (continuous) persistent module usually arises from a weight function. We define a \textit{weighted digraph} as a triple $D= (X , E , A) $ where $X$ is a finite set  with a weight function   $ A: E\to \mathbb{R}  ,  E\subseteq X\times X-\{(x , x) \arrowvert x\in X\}  $.
Let $ A,  A'  $ be two weight functions on E, we define the \textit{distance} between them as $$d_E( A,  A')=max_{(x,y )\in E} \arrowvert  A(x,y)-  A'(x,y)\arrowvert$$ Given a weighted digraph  $D=(X, E,A)  $  and  $ a\in \mathbb{R}   $, define digraphs 
  $G_{a}=(X,E_{a})  $,  $G_{\overline{a}}=(X,E_{\overline{a}})  $  by $$E_{a}:=\{(x,y)\in E : A(x,y)\leq a \}$$ $$E_{\overline{a}}:=\{(x,y)\in E : A(x,y)\geq a \}$$   For any  $ a'\geq a \in \mathbb{R}   $, we have  natural inclusions  $ E_{a}\hookrightarrow E_{a'}   $,  $ E_{\overline{a'}}\hookrightarrow E_{\overline{a}}   $.  The nested families of sub-digraphs $\{G_{a} \}$, $\{G_{\overline{a}} \}$ of $G = (X , E)$ induces nested families of subgroups $\mathcal{A}_*( G_{a})$, $\mathcal{A}_*( G_{\overline{a}})$ of $\mathcal{R}_*(X) $ (see Section 2.2.2 ). Denote the supremum complexes of $ \mathcal{A}_*( G_{*}) $ as $  \mathcal{S}_*( G_{*}) $. Then we have nested families of subcomplexs $\mathcal{S}_*( G_{a})$, $\mathcal{S}_*( G_{\overline{a}})$ of $\mathcal{R}_*(X) $ .Thus for each $p$,  we get a extended persistence module  $V^D_p $ by defining $$V^D_p (a):=H_p(\mathcal{S}_*( G_{a}))=H^{sup}_p(\mathcal{A}_\ast(G_{a});\mathcal{R}_\ast(X))=H_p(G_{a})$$
$$V^D_p (\infty) :=H_p(\mathcal{S}_*( G))=H^{sup}_p(\mathcal{A}_\ast(G);\mathcal{R}_\ast(X))=H_p(G)$$
$$V^D_p (\overline{a}):=H_p(\mathcal{S}_*( G)/\mathcal{S}_*( G_{\overline{a}}))=H^{sup}_p(\mathcal{A}_\ast(G) , \mathcal{A}_\ast(G_{\overline{a}});\mathcal{R}_\ast(X))=H_p(G , G_{\overline{a}})$$ 
for  $a \in \mathbb{R}    $   and $\overline{a} \in \mathbb{R}^o$, and joining them by homomorphisms induced by inclusions. 
\par
Finiteness of $X$ implies that $\mathcal{R}_p (X)$ and $\mathcal{S}_p( G)$ are finite dimensional for all $p\in \mathbb{N}$. Thus $V^D_p$ is finite dimensional everywhere. For sufficiently large $a \in \mathbb{R}$,  $b\in \mathbb{R}$, we have $ G_{a}=G$,  $ G_{\overline{b}}=(X ,\varnothing)$. Thus $(V^D_p)_{s, t}$  is an isomorphism for $a <s<t<\overline{b}$. By Proposition \ref{decomposability}, $V^D_p$ is decomosable for any weighted graph $D$ and any $p\in \mathbb{N} $. 

We are now able to formulate and prove the stability of extended persistent path homology.
\begin{theorem}\label{Stability-weighted digraph}
Let   $ D=(X, E,A^0)  $,  $ D'=(X,  E , A^\delta)  $ be two weighted digraphs with $d_E( A^0,  A^\delta)=\delta$. Then$$d_B(\text{Dgm}_p(V^D_p),\text{Dgm}_p(V_p^{D'}))\leq d_E( A^0,  A^\delta)$$
for any $p\in \mathbb{Z}_+  $
\end{theorem}

\begin{proof}
For any $h\in \{0,\delta\}$, 
  $ a\in \mathbb{R}   $, define $E^h_{a} $, $E^h_{\overline{a}}$ by
\par
 $$E^h_{a}:=\{(x,y)\in E : A^h(x,y)\leq a \}$$$$E^h_{\overline{a}}:=\{(x,y)\in E : A^h(x,y)\geq a \}$$   and define  $G^h_{a}$,  $G^h_{\overline{a}}$ by $ (X, E^h_{a})$, $(X, E^h_{\overline{a}})$ respectively.

 Since $d_E( A^0,  A^{\delta}) =\delta$,  we have the following natural inclusions of digraphs .\\
\[
G^0_{a}\hookrightarrow  G^{\delta}_{a+\delta}   ,
  G^0_{\overline{a}} \hookrightarrow  G^{\delta}_{\overline{a-\delta}}  , 
  G^{\delta}_{a}\hookrightarrow G^0_{a+\delta}   , 
G^{\delta}_{\overline{a}}\hookrightarrow G^0_{\overline{a-\delta}}     \tag{$*$}
\]
   
  Consider now the following commutative diagrams of chain complexes, where all arrows are induced by inclusion or natural projection (the slant arrows are justified by  $( * )$ ):

\par
(i)For any  $ a<a', b<b'\in \mathbb{R}   $, we have diagrams:

\begin{tikzcd}
\mathcal{S}_*( G^0_{a}) \arrow[r] \arrow[rd] \arrow[r] & \mathcal{S}_*( G^0_{a'}) \arrow[rd]       &                                  \\
                                                       & \mathcal{S}_*(  G^{\delta}_{a+\delta}) \arrow[r] & \mathcal{S}_*(  G^{\delta}_{a'+\delta})
\end{tikzcd}
\par
\begin{tikzcd}
\mathcal{S}_*( G^0_{a}) \arrow[r] \arrow[rd] \arrow[r] & \mathcal{S}_*( G)/\mathcal{S}_*( G^0_{\overline{b}}) \arrow[rd] &                                                                  \\
                                                       & \mathcal{S}_*(  G^{\delta}_{a+\delta}) \arrow[r]                     & \mathcal{S}_*( G)/\mathcal{S}_*(  G^{\delta}_{\overline{b-\delta}})
\end{tikzcd}

\begin{tikzcd}
\mathcal{S}_*( G)/\mathcal{S}_*( G^0_{\overline{b'}}) \arrow[r] \arrow[rd] & \mathcal{S}_*( G)/\mathcal{S}_*( G^0_{\overline{b}}) \arrow[rd]         &                                                             \\
                                                                           & \mathcal{S}_*( G)/\mathcal{S}_*(  G^{\delta}_{\overline{b'-\delta}}) \arrow[r] & \mathcal{S}_*(G)/\mathcal{S}_*(  G^{\delta}_{\overline{b-\delta}})
\end{tikzcd}
\par(i') Same diagrams as (i) but with $s, t$ reversed.
\par(ii)For any  $ a, b\in \mathbb{R}   $, we have diagrams:

\begin{tikzcd}
\mathcal{S}_*( G^0_{a}) \arrow[rr] \arrow[rd]                              &                                                                            & \mathcal{S}_*( G^0_{a+2\delta})                              \\
                                                                           & \mathcal{S}_*(  G^{\delta}_{a+\delta}) \arrow[ru]                               &                                                              \\
\mathcal{S}_*( G)/\mathcal{S}_*( G^0_{\overline{b}}) \arrow[rr] \arrow[rd] &                                                                            & \mathcal{S}_*( G)/\mathcal{S}_*( G^0_{\overline{b-2\delta}}) \\
                                                                           & \mathcal{S}_*( G)/\mathcal{S}_*( G^{\delta}_{\overline{b-\delta}}) \arrow[ru] &                                                             
\end{tikzcd} 
\par (ii') Same diagrams as (ii) but with $0, \delta$ reversed.

\par Passing to homology, the above diagrams constitute a $  \delta $-interleaving between $V^{D}_p $ and $V^{D'}_p$. Since all of these modules are decomposable, we have  $d_B(\text{Dgm}_p(V^{D}_p),\text{Dgm}_p(V^{D'}_p)\leq \delta = d_E(A,A') $ by Theorem  \ref{final stablity}.

\end{proof}
\subsection{Stability for Persistent Homology of Hypergraphs}
We shall need the following definition of relative embedded homology for hypergraph pairs. Let $V$ be a finite set and $\mathcal{H'}\subseteq \mathcal{H}$ be  hypergraphs on $V$. Then we have $\Delta_* (\mathcal{H'}) \subseteq \Delta_* (\mathcal{H}) \subseteq\Delta_* (K_{\mathcal{H}})$. Define the relative embedded homology group of $( \mathcal{H}, \mathcal{H'})$ as $H_p ( \mathcal{H}, \mathcal{H'})=H^{sup}_p(\Delta_* (\mathcal{H}), \Delta_* (\mathcal{H'});\Delta_* (K_{\mathcal{H}}))$ (see section \ref{Extended PH for Graded Subgroups}). Note that we could replace $K_{\mathcal{H}}$ by any larger simplicial complex. In particular, an inclusion of hypergraph pairs induces a well-defined homomorphism between their homology groups. 
\par
 Now suppose $\mathcal{H}$ is a hypergraph defined on a finite set $V$, let $f$ be a real valued function on $\mathcal{H}$.
For each $a \in \mathbb{R}$, let
$$\mathcal{H}^f_a:=f^{-1}((-\infty, a])$$
$$\mathcal{H}^f_{\bar{a}}:=f^{-1}([ a, +\infty))$$

 For each $p\in \mathbb{N}$,  we define a extended persistence module  $V^f_p $ by defining 
 
 $$V^f_p (a):=H_p(\mathcal{H}^f_a)=H^{sup}_p(\Delta_* (\mathcal{H}^f_a);\Delta_* (K_{\mathcal{H}}))$$ $$ V^f_p (\infty):=H_p(\mathcal{H})=H^{sup}_p(\Delta_* (\mathcal{H});\Delta_*(K_{\mathcal{H}}))$$
 $$V^f_p (\bar{a}):=H_p(\mathcal{H}, \mathcal{H}^f_{\bar{a}})=H^{sup}_p(\Delta_* (\mathcal{H}), \Delta_* (\mathcal{H}^f_{\bar{a}});\Delta_* (K_{\mathcal{H}}))$$for  any $a \in \mathbb{R}    $   and $\overline{a} \in \mathbb{R}^o$, and joining them by homomorphisms induced by inclusions.  
Suppose $f$, $g$ are two real valued functions on $\mathcal{H}$. Define the $L^{\infty}$ distance between $f$ and $g$  by 
$$\Arrowvert f-g\Arrowvert_{\infty}=\sup_{\sigma \in \mathcal{H}} \arrowvert f(\sigma)-g(\sigma)\arrowvert $$
We are now able to formulate and prove the stability of extended persistent embedded homology of hypergraph. 
\begin{theorem}\label{stability-hypergraph}

Let $f$, $g$ be two real valued functions on $\mathcal{H}$. Then$$d_B(\text{Dgm}_p(V^{f}_p),\text{Dgm}_p(V^{g}_p)) \leq \Arrowvert f-g\Arrowvert_{\infty}$$
\end{theorem}

\begin{sproof}
Let $\delta= \Arrowvert f-g\Arrowvert_{\infty}$ .  Since we have hypergraphs analogous to $(*)$ in the proof of Theorem \ref{Stability-weighted digraph}.  A $\delta$-interleaving between  $V^{f}_p, V^{g}_p$ is then obtained in the same fashion as in Theorem \ref{Stability-weighted digraph} , for we have diagrams analogous to those in the proof of Theorem  
 \ref{Stability-weighted digraph}. Since these modules are decomposable by Proposition \ref{decomposability}, we have $d_B(\text{Dgm}_p(V^{f}_p),\text{Dgm}_p(V^{g}_p) \leq \delta =\Arrowvert f-g\Arrowvert_{\infty}$ by Theorem \ref{final stablity}.
\end{sproof}

 \bibliographystyle{plain}
 \bibliography{DwindledBible}

\begin{thebibliography}{1}

\bibitem{botnan2020decomposition}
Magnus Botnan and William Crawley-Boevey.
\newblock Decomposition of persistence modules.
\newblock {\em Proceedings of the American Mathematical Society},
  148(11):4581--4596, 2020.

\bibitem{HypergraphHomology}
S.~Bressan, Jingyan Li, S.~Ren, and Jie Wu.
\newblock The embedded homology of hypergraphs and applications.
\newblock {\em Asian Journal of Mathematics}, 23(3):479 – 500, 2019.

\bibitem{chazal2016structure}
Fr{\'e}d{\'e}ric Chazal, Vin De~Silva, Marc Glisse, and Steve Oudot.
\newblock {\em The structure and stability of persistence modules}.
\newblock Springer, 2016.

\bibitem{chen2011persistent}
Chao Chen and Michael Kerber.
\newblock Persistent homology computation with a twist.
\newblock In {\em Proceedings 27th European workshop on computational
  geometry}, volume~11, pages 197--200, 2011.

\bibitem{Chowdhury18}
Samir Chowdhury and Facundo M\'{e}moli.
\newblock Persistent path homology of directed networks.
\newblock In {\em Proceedings of the Twenty-Ninth Annual ACM-SIAM Symposium on
  Discrete Algorithms}, SODA '18, page 1152–1169, USA, 2018. Society for
  Industrial and Applied Mathematics.

\bibitem{CohenSteiner2009ExtendingPH}
David Cohen-Steiner, Herbert Edelsbrunner, and John Harer.
\newblock Extending persistence using poincar{\'e} and lefschetz duality.
\newblock {\em Foundations of Computational Mathematics}, 9:79--103, 2009.

\bibitem{CT}
Herbert Edelsbrunner and John Harer.
\newblock {\em Computational Topology - an Introduction.}
\newblock American Mathematical Society, 2010.

\bibitem{grigor2020path}
AA~Grigor’yan, Yong Lin, Yu~V Muranov, and Shing-Tung Yau.
\newblock Path complexes and their homologies.
\newblock {\em Journal of Mathematical Sciences}, 248(5):564--599, 2020.

\bibitem{ren2020stability}
Shiquan Ren and Jie Wu.
\newblock Stability of persistent homology for hypergraphs.
\newblock {\em arXiv preprint arXiv:2002.02237}, 2020.

\end{thebibliography}
 
\end{document}